\documentclass[10pt,oneside,a4paper]{amsart}
\usepackage{amsmath,xy,latexsym,amsthm,verbatim}
\usepackage[psamsfonts]{amssymb}
\usepackage[applemac]{inputenc}
\usepackage{fancyhdr}
\usepackage{mathrsfs}
\usepackage[bookmarks,colorlinks]{hyperref}
\usepackage{emp}
\usepackage{tikz}
\usetikzlibrary{decorations.markings}

\fancypagestyle{plain}{
	\fancyhead{}
	
}

\theoremstyle{definition}
\newtheorem{Def}{Definition}[section]
\theoremstyle{plain}
\newtheorem{Thm}{Theorem}[section]
\newtheorem{Lem}[Thm]{Lemma}

\newtheorem{Prop}[Thm]{Proposition}

\newtheorem{rem}{Remark}

\newtheorem{Ass}{Assumptions}

\newcommand{\R}{\mathbb{R}}

\newcommand{\C}{\mathbb{C}}
\newcommand{\Rn}{\mathbb{R}^n}

\newcommand{\N}{\mathbb{N}}

\newcommand{\m}{\mathbf{m}}

\newcommand{\x}{\langle x\rangle}
\newcommand{\csi}{\langle \xi \rangle}
\newcommand{\cl}{\textnormal{cl}}

\newcommand{\op}{\textnormal{Op}}

\newcommand{\Res}{\mathrm{Res}}
\newcommand{\simb}{\textnormal{sym}}

\renewcommand{\Re}{\mathrm{Re}}

\def\SX{ {\mathcal{S}} }

\def\Op#1{ {\mathrm{Op}} \left( #1 \right) }

\def\Symp#1{ {\mathrm{Sym_{p}}} \left( #1 \right) }

\def\norm#1{ \left< #1 \right> }

\DeclareMathOperator{\TR}{\mathrm{TR}}

\newcommand{\biindice}[3]%
{
\begin{array}[t]{c}
#1\\
{\scriptstyle #2}\\
{\scriptstyle #3}
\end{array}
}

\author{Ubertino Battisti}

\address{Dipartimento di Matematica, Università di Torino, Italy}

\email{ubertino.battisti@unito.it}

\author{Sandro Coriasco}

\address{Dipartimento di Matematica, Università di Torino, Italy}

\email{sandro.coriasco@unito.it}

\title{Wodzicki Residue for\\ Operators on Manifolds with Cylindrical Ends}

\keywords{Wodzicki Residue, $SG$-calculus, Manifold with Cylindrical Ends, Weyl formula}

\subjclass[2000]{Primary: 58J40; Secondary: 58J42, 47A10, 47G30, 47L15}

\begin{document}

\begin{abstract}
We define the Wodzicki Residue $\TR(A)$ for $A$ belonging to a space of operators with double order, denoted  $L^{m_1,m_2}_\cl$. Such operators are globally defined initially on $\R^n$ and then, more generally, on a class of non-compact manifolds, namely, the manifolds with cylindrical ends. The definition is based on the analysis of the associate zeta function $\zeta(A,z)$. Using this approach, under suitable ellipticity assumptions, we also compute a two terms leading part of the Weyl formula for a positive selfadjoint operator $A\in L^{m_1,m_2}_\cl$ in the case $m_1=m_2$.
\end{abstract}

\maketitle

\section*{Introduction}

The aim of this paper is to extend the definition of Wodzicki Residue to the class of the so-called
$SG$-classical operators on manifolds with cylindrical ends, through the analysis of their complex powers and of the associate zeta functions. In view of the properties of the underlying calculus, our definition holds for $SG$-classical operators $A$ globally defined on $\R^n$ as well as for their counterparts globally defined on manifolds with cylindrical ends.
Under appropriate conditions, the information so obtained, concerning the meromorphic structure of 
$\zeta(A,z)$, is precise enough to allow us to improve known results about the Weyl formula for the eigenvalue asymptotics of elliptic $SG$-classical operators.

More explicitly, $SG$-pseudodifferential operators $A=a(x,D)=\Op{a}$ can be defined via the usual left-quantization
\[
	Au(x)= \frac{1}{(2 \pi)^{n}} \int e^{i x \cdot \xi} a(x, \xi) \hat u(\xi) d\xi,\quad u\in\SX(\R^n),
\]
starting from symbols $a(x,\xi) \in C^\infty(\R^n\times\R^n)$ with the property that, for arbitrary multiindices $\alpha,\beta$, there exist constants $C_{\alpha\beta}\ge0$ such that the estimates 
\begin{equation}
	\label{disSG}
	|D_\xi^{\alpha}D_x^{\beta} a(x, \xi)| \leq C_{\alpha\beta} \langle\xi\rangle^{m_1-|\alpha|}\langle x\rangle^{m_2-|\beta|}
\end{equation}
hold for fixed $m_1,m_2\in\R$ and all $(x, \xi) \in \R^n \times \R^n$, where $\langle u \rangle=\sqrt{1+|u|^2}$, $u\in\R^n$.
Symbols of this type belong to the class denoted by $SG^{m_1,m_2}(\Rn)$, and the corresponding operators constitute the class
$L^{m_1,m_2}(\Rn)=\Op{SG^{m_1,m_2}(\Rn)}$. In the sequel we will often simply write $SG^{m_1,m_2}$ and $L^{m_1,m_2}$, respectively,
fixing the dimension of the (non-compact) base manifold to $n$. 

These classes of operators were first introduced on $\R^n$ by H.O.~Cordes \cite{CO} and C.~Parenti \cite{PA72}, see also R.~Melrose \cite{ME}. They form a graded algebra, i.e., $L^{r_1,r_2}\circ L^{m_1,m_2}\subseteq L^{r_1+m_1,r_2+m_2}$, whose residual elements are operators with symbols in $\displaystyle SG^{-\infty}(\R^n)= \bigcap_{(m_1,m_2) \in \R^2} SG^{m_1,m_2} (\R^n)=\SX(\R^{2n})$, that is, those having kernel in $\SX(\R^{2n})$, continuously mapping $\SX^\prime(\R^n)$ to $\SX(\R^n)$. An operator $A=\Op{a}\in L^{m_1,m_2}$ is called $SG$-elliptic if there exists $R\ge0$ such that $a(x,\xi)$ is invertible for $|x|+|\xi|\ge R$ and
\[
	a(x,\xi)^{-1}=O(\norm{\xi}^{-m_1}\norm{x}^{-m_2}).
\] 
Operators in $L^{m_1,m_2}$ act continuously from $\SX(\R^n)$ to itself, and extend as continuous operators from $\SX^\prime(\R^n)$ to itself and from $H^{s_1,s_2}(\R^n)$ to $H^{s_1-m_1,s_2-m_2}(\R^n)$, where $H^{t_1,t_2}(\R^n)$, $t_1,t_2\in\R$, denotes the weighted Sobolev space
\begin{align*}
  	H^{t_1,t_2}(\R^n)&= \{u \in \SX^\prime(\R^{n}) \colon \|u\|_{t_1,t_2}= \|\Op{\pi_{t_1,t_2}}u\|_{L^2}< \infty\},
 	\\
  	\pi_{t_1,t_2}(x,\xi)&= \langle \xi \rangle^{t_1} \langle x\rangle^{t_2}.
\end{align*}
Incidentally, note that $H^{s_1,s_2}(\R^n)\hookrightarrow H^{r_1,r_2}(\R^n)$ when $s_1\ge r_1$ and $s_2\ge r_2$, with compact embedding when both inequalities are strict, while
\[
	\displaystyle\SX(\R^n)=\bigcap_{(s_1,s_2) \in \R^2} H^{s_1,s_2}(\R^n)
	\mbox{ and }
	\displaystyle\SX^\prime(\R^n)=\bigcup_{(s_1,s_2) \in \R^2} H^{s_1,s_2}(\R^n).
\]
An elliptic $SG$-operator $A \in L^{m_1,m_2}$ admits a parametrix $P\in L^{-m_1,-m_2}$ such that
\[
PA=I + K_1, \quad AP= I+ K_2,
\]
for suitable $K_1, K_2 \in L^{-\infty}$, and it turns out to be a Fredholm operator.
In 1987, E.~Schrohe \cite{SC87} introduced a class of non-compact manifolds, the so-called $SG$-manifolds, on which it is possible to transfer from $\R^n$ the whole $SG$-calculus: in short, these are manifolds which admit a finite atlas whose changes of coordinates behave like symbols of order $(0,1)$ (see \cite{SC87} for details and additional technical hypotheses). The manifolds with cylindrical ends are a special case of $SG$-manifolds, on which also the concept of $SG$-classical operator makes sense: moreover, the principal symbol of a $SG$-classical operator $A$ on a manifold with cylindrical ends $M$, in this case a triple $\sigma(A)=(\sigma_\psi(A),\sigma_e(A),\sigma_{\psi e}(A))$, has an invariant meaning on $M$, see Y.~Egorov and B.-W.~Schulze \cite{ES97}, L.~Maniccia and P.~Panarese \cite{MP02}, R.~Melrose \cite{ME} and Section \ref{cylexit} below.

Wodzicki Residue was first considered by M.~Wodzicki in 1984 \cite{WO84}, in the setting of pseudodifferential operators on closed manifold, while studying the meromorphic continuation of the zeta function of elliptic operators: the latter had been originally defined by R.~Seeley \cite{SE67}. Wodzicki Residue turns out to be a trace on the algebra of classical operators modulo smoothing operators. Moreover, if the dimension of the closed manifold is larger then one, it is the unique trace on such algebra, up to multiplication by a constant 
(the situation in dimension one is different, as a consequence of the fact that, in such a case, $S^*M$ is not connected, cfr. C. Kassel \cite{KA89}). 
In 1985, V.~Guillemin \cite{GU85} independently defined the so-called Symplectic Residue, equivalent to Wodzicki Residue, with the aim
of ``finding a \emph{soft} proof of Weyl formula'', see also \cite{LJ10} for an extension of techniques used by V. Guillemin. For an overview on the subject, see also the monograph of S. Scott \cite{SC10}.
Wodzicki Residue, sometimes called non-commutative trace, gained a growing interest in the years, also in view of the links with non-commutative geometry and Dixmier trace, see, e.g., A. Connes \cite{CO88}, B. Ammann and C. B\"ar \cite{AB02}, W. Kalau and M. Walze \cite{KW95}, D. Kastler \cite{KA95}, R. Ponge \cite{PO08}, U. Battisti and S. Coriasco \cite{BC10}. The concept has been extended to different situations: manifolds with boundary by B.V. Fedosov, F. Golse, E. Leichtnam and E. Schrohe \cite{FE96}, conic manifolds by E. Schrohe \cite{SC97} and J.B. Gil and P.A. Loya \cite{LO02}, operators with log-polyhomogeneous symbols by M. Lesch \cite{Le99}, anisotropic operators on $\R^n$ by P. Boggiatto and F. Nicola \cite{BN03}, Heisenberg Calculus by R. Ponge \cite{PO07}, holomorphic families of psedudofferential operators by S. Paycha and S. Scott  \cite{PS07}.
Wodzicki Residue in the case of $SG$-calculus on $\R^n$ was defined by F. Nicola \cite{NI03}, with an approach which differs by the one used here.

Our purpose is to show the relation between Wodzicki Residue $\TR(A)$ and the zeta function 
$\zeta(A,z)$ of elliptic $SG$-classical operators $A\in L^{m_1,m_2}_\cl$ of integer order,
on manifolds with cylindrical ends. Under suitable assumptions, see Section \ref{sec:zpower} below, it turns out that $\zeta(A,z)$ is holomorphic for $\Re(-z)$ big enough, and that it can be extended as a meromorphic function on the whole complex plane, with poles of order at most two. Then, we define the Wodzicki Residue of $A$ as 
\[
\TR(A)=m_1m_2\mathrm{Res}^2_{z=1}(\zeta(A,z))=m_1m_2\lim_{z \to 1} (z-1)^2 \zeta(A,z).
\]
We prove in Theorem \ref{thm:equiv} that this definition agrees, for operators on $\R^n$, with the one given in \cite{NI03}. Our viewpoint, compared with the approach by F. Nicola, looks more convenient when dealing with $SG$-classical operators on manifolds with cylindrical ends, on which we are  focused. $\TR(A)$ is then extended to general $SG$-classical operators with integer order. In order to evaluate the residue of $\zeta(A,z)$ at $z=1$, that is
\[
\begin{split}
\widehat{\TR}_{x,\xi}(A)&= \lim_{z\to 1} (z-1) \Big[\zeta(A, z)-
 \frac{\Res^2_{z=1}(\zeta(A,z))}{(z-1)^2}\Big]\\
 &=-  \frac{1}{m_1}\widehat{\textnormal{Tr}}_\psi(A) - \frac{1}{m_2} \widehat{\textnormal{Tr}}_e(A) + \frac{1}{m_1 m_2} \widehat{\TR}_{\theta}(A),
\end{split}
\]
we obtain the functionals $\widehat{\textnormal{Tr}}_\psi, \widehat{\textnormal{Tr}}_e$, introduced in \cite{NI03}, together with the additional functional
$\widehat{\TR}_\theta(A)$, that we have called the \textit{angular term}. $\widehat{\TR}_\theta(A)$ plays an essential role in the study of the eigenvalue asymptotics of suitable elliptic $SG$-classical operators: our results, in fact, give an alternative proof of the corresponding Weyl formulae, on $\R^n$ as well as on manifolds with cylindrical ends, see F. Nicola \cite{NI03} and L. Maniccia and P. Panarese \cite{MP02}, respectively. Moreover, in one case we obtain a more precise formula.
Indeed, let $A \in L^{m_1, m_2}_\cl(\R^n)$ be a positive selfadjoint elliptic $SG$-classical operator of order $(m_1, m_2)$, $m_1,m_2$ positive integers, and let $\displaystyle N_A(\lambda)=\sum_{\lambda_j \le \lambda}1$ denote the associate counting function, $\{\lambda_j\}_{j\ge1}$ being the (non-decreasing) sequence of the eigenvalues of $A$. Then, if $A$ satisfies the hypotheses mentioned above, for certain $\delta_i>0$, $i=0,1,2$, and $\lambda\to+\infty$,
\begin{equation}
\label{weylnor}
\begin{split}
N_A(\lambda) =& C^1_0 \lambda^{\frac{n}{m}}\log \lambda + C^2_0 \lambda^{\frac {n}{m}}+ O(\lambda^{\frac{n}{m}-\delta_0})  \quad \mbox{for } m_1=m_2=m, \\
N_A(\lambda) =& C_1 \lambda^{\frac{n}{m_1}}+ O(\lambda^{\frac{n}{m_1}-\delta_1})  \quad \mbox{for } m_1<m_2,\\
N_A(\lambda) =& C_2 \lambda^{\frac{n}{m_2}}+ O(\lambda^{\frac{n}{m_2}-\delta_2}) \quad \mbox{for }m_1>m_2,
\end{split}
\end{equation}
where the constants $C^1_0, C^2_0, C_1, C_2$ depend on $A$, see \eqref{constweyla}-\eqref{constweylc} below.
There are others setting in which the counting function $N(\lambda)$ has an asymptotic behaviour of type $\lambda^m \log (\lambda)$, see, for example, \cite{LO02}, \cite{MO08}, \cite{BA10}.
Christiansen and Zworski \cite{CZ95} studied the Weyl asymptotics of the counting function of the Laplacian on manifolds with cylindrical ends, while here we focus on operators with discrete spectrum.

In Theorem \ref{weylrn} we prove that if $A \in L^{m_1, m_2}_\cl(M)$ is a $SG$-classical operator on a manifold with cylindrical ends $M$, fulfilling assumptions completely similar to those required for operators on 
$\R^n$, \eqref{weylnor} holds as well. While the asymptotic behaviours of $N_A(\lambda)$ in \eqref{weylnor} already appeared in such form for the cases $m_1 \not= m_2$, we can explicitely write the second term in the case $m_1=m_2=m$: only the logarithmic term in $\lambda$ is present in the mentioned papers \cite{MP02, NI03}, and, to the best knowledge of the authors, the fact that the second term is the power
$C^2_0 \lambda^{\frac {n}{m}}$ has not been proved elsewhere for the operators we consider. 

The paper is organised as follows. In Section \ref{sec:zpower} we review some of the standard facts about the calculus of $SG$-classical operators. To make our exposition more self-contained, we also sketch the construction of the complex powers of a $SG$-classical elliptic operator $A$ under suitable hypotheses. Existence and properties of $A^z$, $z\in\C$, were proved by E. Schrohe, J. Seiler and L. Maniccia in \cite{MSS06}, to which we mainly refere. However, we also give a different proof of the fact that $A^z$ is still a $SG$-classical operator, using the isomorphism between the classical $SG$-symbols and the smooth functions on the product of two $n$-dimensional closed balls. In Section \ref{seczeta} we proceed with the study of the function $\zeta(A,z)$ and the definition of Wodzicki Residue of $A$. Finally, in Section \ref{cylexit} we extend our results to the case of $SG$-classical elliptic operators on manifolds with cylindrical ends and prove the asymptotic expansions \eqref{weylnor}. 

\section*{Acknowledgements}
The authors wish to thank F. Nicola and L. Rodino for useful discussions and suggestions.

%
%

\section{\texorpdfstring{Complex powers of $SG$-classical operators}{Complex powers of $SG$-classical operators}}
\label{sec:zpower}
From now on, we will be concerned with the subclass of operators given by those elements $A\in L^{m_1,m_2}(\R^n)$, 
$(m_1,m_2)\in\R^2$, which are $SG$-classical, that is, $A=\Op{a}$ with $a\in SG^{m_1,m_2}_\cl(\R^n)\subset SG^{m_1,m_2}(\R^n)$. We begin recalling the basic definitions and results (see, e.g., \cite{ES97,MSS06} for additional details and proofs).
\setcounter{equation}{0}

\begin{Def}
\label{def:sgclass-a}
\begin{itemize}
\item[i)]A symbol $a(x, \xi)$ belongs to the class $SG^{m_1,m_2}_{\cl(\xi)}(\R^n)$ if there exist $a_{m_1-i, \cdot} (x, \xi)\in \widetilde{\mathscr{H}}_\xi^{m_1-i}(\R^n)$, $i=0,1,\dots$, homogeneous functions of order $m_1-i$ with respect to the variable $\xi$, smooth with respect to the variable $x$, such that, for a $0$-excision function $\omega$,
\[
a(x, \xi) - \sum_{i=0}^{N-1}\omega(\xi) \, a_{m_1-i, \cdot} (x, \xi)\in SG^{m_1-N, m_2}(\R^n), \quad N=1,2, \ldots;
\]
\item[ii)]A symbol $a(x, \xi) $ belongs to the class $SG_{\cl(x)}^{m_1,m_2}(\R^n)$ if there exist $a_{\cdot, m_2-k}(x, \xi)\in \widetilde{\mathscr{H}}_x^{m_2-k}(\R^n)$, $k=0,\,\dots$, homogeneous functions of order $m_2-k$ with respect to the variable $x$, smooth with respect to the variable $\xi$, such that, for a $0$-excision function $\omega$,
\[
a(x, \xi)- \sum_{k=0}^{N-1}\omega(x) \, a_{\cdot, m_2-k}(x,\xi) \in SG^{m_1, m_2-N}(\R^n), \quad N=1,2, \ldots
\]
\end{itemize}
\end{Def}
\begin{Def}
\label{def:sgclass-b}
A symbol $a(x,\xi)$ is $SG$-classical, and we write $a \in SG_{\cl(x,\xi)}^{m_1,m_2}(\R^n)=SG_{\cl}^{m_1,m_2}(\R^n)=SG_{\cl}^{m_1,m_2}$, if
\begin{itemize}
\item[i)] there exist $a_{m_1-j, \cdot} (x, \xi)\in \widetilde{\mathscr{H}}_\xi^{m_1-j}(\R^n)$ such that, 
for a $0$-excision function $\omega$, $\omega(\xi) \, a_{m_1-j, \cdot} (x, \xi)\in SG_{\cl(x)}^{m_1-j, m_2}(\R^n)$ and
\[
a(x, \xi)- \sum_{j=0}^{N-1} \omega(\xi) \, a_{m_1-j, \cdot}(x, \xi) \in SG^{m_1-N, m_2}(\R^n), \quad N=1,2,\dots;
\]
\item[ii)] there exist $a_{\cdot, m_2-k}(x, \xi)\in \widetilde{\mathscr{H}}_x^{m_2-k}(\R^n)$ such that, 
for a $0$-excision function $\omega$, $\omega(x)\,a_{\cdot, m_2-k}(x, \xi)\in SG_{\cl(\xi)}^{m_1, m_2-k}(\R^n)$ and
\[
a(x, \xi) - \sum_{k=0}^{N-1} \omega(x) \, a_{\cdot, m_2-k} \in SG^{m_1, m_2-N}(\R^n), \quad N=1,2,\dots
\] 
\end{itemize}
We set $L_{\cl(x, \xi)}^{m_1,m_2}(\R^n)=L_{\cl}^{m_1,m_2}=\Op{SG^{m_1,m_2}_{\cl}}$.
\end{Def}

\begin{rem}
	The definition could be extended in a natural way
	from operators acting between scalars to operators acting between (distributional sections of) 
	vector bundles: one should then use matrix-valued symbols whose entries satisfy the estimates \eqref{disSG}
	and modify accordingly the various statements below.
	To simplify the presentation, we omit everywhere any reference to vector bundles, assuming them to be trivial and one-dimensional. 
\end{rem}

\noindent
The next two results are especially useful when dealing with $SG$-classical symbols.

\begin{Thm}
	\label{thm4.6}
	Let $a_{k} \in SG_{\cl}^{m_1-k,m_2-k}$, $k =0,1,\dots$,
	be a sequence of $SG$-classical symbols
	and $a \sim \sum_{k=0}^\infty a_{k}$
	its asymptotic sum in the general $SG$-calculus.
	Then, $a \in SG_{\cl}^{m_1,m_2}$.
\end{Thm}

\begin{Thm}
	\label{thm4.6.1.1}
	Let $\mathbb{B}^n= \{ x \in \R^n: |x| \le 1 \}$  and let $\chi$ be a 
	diffeomorphism from the interior of $\mathbb{B}^n$ to $\R^{n}$ such that
	\[ 
	\chi(x) = \displaystyle \frac{x}{|x|(1-|x|)} \quad \mbox{for} \quad |x| > 2/3.
	\]
	Choosing a smooth function $[x]$ on $\R^n$ such that
	$1 - [x] \not= 0$ for all $x$ in the interior of $\mathbb{B}^n$
	and $|x| > 2/3 \Rightarrow [x] = |x|$,
	for any $a \in SG^{m_1,m_2}_{\cl}$ denote by $(D^\m a)(y, \eta)$,
	$\m=(m_1,m_2)$, the function
	\begin{equation}
		\label{eq4.26.1}
		b(y,\eta) = (1-[\eta])^{m_{1}} (1-[y])^{m_{2}}
		a(\chi(y), \chi(\eta)).
	\end{equation}
	Then, $D^\m$ extends to a homeomorphism from $SG^{m_1,m_2}_{\cl}$ to 
	$C^\infty(\mathbb{B}^n \times \mathbb{B}^n)$.
\end{Thm}

\noindent
Note that the definition of $SG$-classical symbol implies a condition of compatibility for the terms of the expansions with respect to $x$ and $\xi$. In fact, defining $\sigma_\psi^{m_1-j}$ and $\sigma_e^{m_2-i}$ on $SG_{\cl(\xi)}^{m_1,m_2}$ and $SG_{\cl(x)}^{m_1,m_2}$, respectively,  as
\begin{align*}
	\sigma_\psi^{m_1-j}(a)(x, \xi) &= a_{m_1-j, \cdot}(x, \xi),\quad j=0, 1, \ldots, 
	\\
	\sigma_e^{m_2-i}(a)(x, \xi) &= a_{\cdot, m_2-i}(x, \xi),\quad i=0, 1, \ldots,
\end{align*}
it possibile to prove that
\[
\begin{split}
a_{m_1-j,m_2-i}=\sigma_{\psi e}^{m_1-j,m_2-i}(a)=\sigma_\psi^{m_1-j}(\sigma_e^{m_2-i}(a))= \sigma_e^{m_2-i}(\sigma_\psi^{m_1-j}(a)), \\
j=0,1, \ldots, \; i=0,1, \ldots
\end{split}
\]
Moreover, the algebra property of $SG$-operators and Theorem \ref{thm4.6} imply that the composition of two $SG$-classical operators is still classical. 
For $A=\Op{a}\in L^{m_1,m_2}_\cl$ the triple $\sigma(A)=(\sigma_\psi(A),\sigma_e(A),\sigma_{\psi e}(A))=(a_{m_1,\cdot}\,,\,
a_{\cdot,m_2}\,,\, a_{m_1,m_2})$ is called the \textit{principal symbol of $A$}. This definition keeps the usual multiplicative behaviour, that is, for any $A\in L^{r_1,r_2}_\cl$, $B\in L^{s_1,s_2}_\cl$, $(r_1,r_2),(s_1,s_2)\in\R^2$,
$\sigma(AB)=\sigma(A)\,\sigma(B)$, with componentwise product in the right-hand side. We also set 
\begin{align*}
	\Symp{A}(x,\xi) = &\;\,\Symp{a}(x,\xi) =\\
	=
	&\;\, a_\m(x,\xi)=\omega(\xi) a_{m_1,\cdot}(x,\xi) +
	\omega(x)(a_{\cdot,m_2}(x,\xi) - \omega(\xi) a_{m_1,m_2}(x,\xi))
\end{align*}
for a $0$-excision function $\omega$. Theorem \ref{thm:ellclass} below allows to express the ellipticity
of $SG$-classical operators in terms of their principal symbol:
\begin{Thm}
	 \label{thm:ellclass}
	An operator $A\in L^{m_1,m_2}_\cl$ is elliptic if and only if each element of the triple $\sigma(A)$ is 
	non-vanishing on its domain of definition.
\end{Thm}

Next, in order to outline the construction of the complex powers $A^z$, $z\in\C$, of an elliptic $SG$-classical operator, we need to recall some results about the parametrix of $A-\lambda I$, where $\lambda$  belongs to a sector $\Lambda$ in the complex plane. It is well-known (see \cite{GS95}) that, unless one restricts to differential operators, it is not possible to follow the idea of \cite{SH87}, that is, define parameter-dependent symbols and analyze the resolvent as a particular case of parameter-dependent operator. We follow the paper \cite{MSS06} closely, recalling facts concerning the spectrum of elliptic $SG$-operators and the concept of $\Lambda$-elliptic symbol. For $m_1,m_2>0$, an elliptic
$A\in L^{m_1,m_2}$ is considered as an unbounded operator 
\[
	A\colon H^{m_1,m_2}\subset L^2\to L^2,
\]
which turns out to be closed.

\begin{Thm}
\label{disspe}
Given an elliptic operator $A \in L^{m_1,m_2}$ with $m_1, m_2>0$,
only one of the following properties holds:
\begin{itemize}
\label{dispe}
\item[i)] the spectrum of $A$ is the whole complex plane $\C$;
\item[ii)] the spectrum of $A$ is a countable set, without any limit point.
\end{itemize}
\end{Thm}
\begin{proof}
If i) does not hold, there exists $\mu \in \C$ such that $(A- \mu I)$ is invertible. Without loss of generality, we can assume $\mu=0$, so that
\[
(A-\lambda I)= A(I-\lambda A^{-1}),
\]
showing that $(A-\lambda I)$ is not invertible if and only if $\lambda\not=0$ and $\frac{1}{\lambda}$ belongs to the spectrum of 
$A^{-1}$. From the properties of elliptic operators, we have that $A^{-1} \in L^{-m_1,-m_2}$. Moreover, in view of the hypothesis $m_1,m_2>0$, of the continuity of $A^{-1}$ from $L^2 \equiv H^{0,0}$ to $H^{m_1,m_2}$ and of the compact embeddings between the weighted Sobolev spaces stated above, $A^{-1}\colon L^2\to H^{m_1,m_2}\hookrightarrow L^2$ is a compact operator, and, as such, it has a countable spectrum with, at most, the origin as a limit point.
\end{proof}
\begin{rem}
The proof of Theorem \ref{disspe} also shows that the eigenfunctions of $A$ are the same of $A^{-1}$.
\end{rem}

\noindent
For fixed $\theta_0,\theta$, let $\Lambda=\{z \in \C \colon \theta_0-\theta \leq \textnormal{arg}(z) \leq \theta_0+\theta\}$ be a closed sector of the complex plane with vertex at the origin. We now recall the definition of ellipticity with respect to $\Lambda$:
\begin{Def}
Let $\Lambda$ be a closed sector of the complex plane with vertex at the origin. A symbol $a(x, \xi) \in SG^{m_1,m_2}$ and the corresponding operator
$A=\Op{a}$ are called $\Lambda$-elliptic if there exist constants $C,R > 0$ such that
\begin{itemize}
\item[i)] $a(x,\xi)-\lambda \not=0$, for any $\lambda \in \Lambda$ and $(x,\xi)$ satisfying $|x|+|\xi|\geq R$;
\item[ii)] $|(a(x,\xi)-\lambda)^{-1}|\leq C  \langle \xi \rangle^{-m_1} \langle x\rangle^{-m_2}$ for any $\lambda \in \Lambda$
 and $(x,\xi)$ satisfying $|x|+|\xi|\geq R$. 
\end{itemize}
\end{Def}
\begin{rem}
\label{matrixvalu}
When matrix-valued symbols are involved, condition i) above is modified, asking that the spectrum of the matrix $a(x,\xi)$ does not
intersect the sector $\Lambda$ for $|x|+|\xi|\geq R$.
\end{rem}
\noindent To define the complex powers of an elliptic operator $A$, we need that the resolvent $(A-\lambda I)^{-1}$ exists, at least, for $|\lambda|$ big enough. The following Theorem \ref{sect} shows that this is always the case when $m_1,m_2>0$ and that the resolvent can be
well approximated by a parametrix of $A-\lambda I$.

\begin{Thm}
\label{sect}
Let $m_1,m_2>0$ and $A \in L^{m_1,m_2}$ be $\Lambda$-elliptic. Then,
there exists a constant $L$ such that the resolvent set $\rho(A)$ includes $\Lambda_L=\{\lambda \in \Lambda \colon |\lambda|>L\}$. Moreover, for suitable constants $C, C^\prime>0$, we have that
\[
\|(A-\lambda I)^{-1}\|_{\mathcal{L}(L^2)}\leq \frac{C}{|\lambda|}
\]
and
\[
\|(A-\lambda I)^{-1}-B(\lambda)\|_{\mathcal{L}(L^2)}\leq \frac{C'}{|\lambda|^2}
\]
where $B(\lambda)$ is a parametrix of $A-\lambda I$.
\end{Thm}

\noindent
The next two results give estimates for the position of the eigenvalues of a $\Lambda$-elliptic operator in the complex plane and the relation between 
$\Lambda$-ellipticity and the principal symbol of a classical $SG$-operator, similarly to Theorem \ref{thm:ellclass}.

\begin{Lem}
\label{elllemma}
If $A=\Op{a}\in L^{m_1,m_2}$, $m_1,m_2>0$, is $\Lambda$-elliptic, there is a constant $c_0\geq 1$ such that, for every $(x,\xi) \in \R^n\times \R^n$, the spectrum of $a(x,\xi)$ is
included in the set
\[
	\Omega_{\x, \csi}=\left\{z \in \C\setminus \Lambda \colon \frac{1}{c_0}\csi^{m_1}\x^{m_2}\leq |z|\leq c_0 \csi^{m_1}\x^{m_2}\right\}
\]
and
\[
|(\lambda- a(x,\xi))^{-1}|\leq C(|\lambda| +\csi^{m_1}\x^{m_2} )^{-1}, \quad \forall (x,\xi) \in \R^n \times \R^n, \lambda \in \C\setminus \Omega_{\x, \csi}.
\]
\end{Lem}

\begin{Prop}
\label{prop:lambdaell}
For $a \in SG_{\cl}^{m_1,m_2}$, the $\Lambda$-ellipticity property is equivalent to
\[
\begin{split}
a_{m_1,\cdot}(x,\omega)-\lambda & \not=0,\, \mbox{for all } x \in \R^n,\, \omega \in \mathbb{S}^{n-1}, \,\lambda \in \Lambda,\\
a_{\cdot, m_2}(\omega',\xi)-\lambda &\not= 0, \, \mbox{for all } \xi \in \R^n, \,\omega' \in \mathbb{S}^{n-1}, \,\lambda \in \Lambda,\\
a_{m_1, m_2}(\omega', \omega)-\lambda &\not=0, \, \mbox{for all } \omega \in \mathbb{S}^{n-1},  \,\omega'\in \mathbb{S}^{n-1}, \,\lambda \in \Lambda,
\end{split}
\]
where $\mathbb{S}^{n-1}=\{u \in \R^n \colon |u|=1\}$.
\end{Prop}
\begin{rem}
If $a$ is matrix-valued, the conditions in Proposition \ref{prop:lambdaell} have to be expressed in terms of the spectra of the three involved matrices, analogously to Remark \ref{matrixvalu}.
\end{rem}
We can now give the definition of $A^z$, $z\in\C$. The following assumptions on $A$ are natural:
\renewcommand{\theAss}{A\arabic{Ass}}
\begin{Ass}
\label{assu}
\begin{enumerate}
\item $A \in L^{m_1, m_2}(\R^n)$, with $m_1$ and $m_2$ positive integers;
\item $A$ is $\Lambda$-elliptic with respect to a closed sector $\Lambda$ of the complex plane with vertex at the origin;
\item $A$ is invertible as an operator from $L^2(\R^n)$ to itself;
\item The spectrum of $A$ does not intersect the real interval $(-\infty, 0)$;
\item $A$ is $SG$-classical.
\end{enumerate}
\end{Ass}

\noindent
Theorem \ref{dispe} implies that, if $A$ satisfies Assumptions \ref{assu}, it has a discrete spectrum. In view of this, it is possible to find $\theta\in(0,\pi)$ so that $(A- \lambda)^{-1}$ exists for all $\lambda \in  \Lambda=\Lambda(\theta)= \{z \in \C \colon \pi-\theta\leq \arg (z)\leq \pi+\theta \}$.
\begin{Def}
Let $A$ be a $SG$-operator that satisfies Assumptions \ref{assu}. Let us define $A_z$, $z\in\C$, $\Re(z)<0$, as 
\begin{equation}
\label{Az}
A_z= \frac{1}{2 \pi i} \int_\Gamma \lambda^z (A-\lambda I)^{-1} d\lambda,
\end{equation} 
where $\Gamma=\partial^-(\Lambda \cup \{z\in\C \colon |z|\le \delta\})$, $\delta>0$ suitably small, is the path in the figure below: 
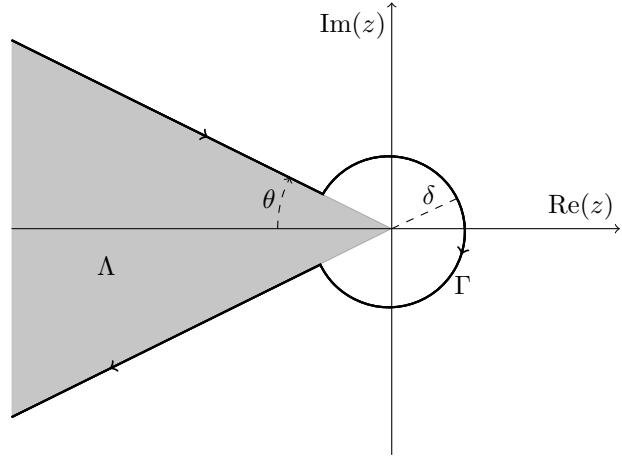
\begin{center}
	\begin{tikzpicture}[xscale=1,yscale=1]
		\draw[->] (1,4)--(9,4);
		\draw[->] (6,1)--(6,7);
		\filldraw[nearly transparent, color=black!90!white] 
		(1,6.5) -- (6,4) -- (1,1.5);
		\begin{scope}[thick, decoration={markings,
    			mark=at position 0.2 with {\arrow{>}}}]	
			\draw[postaction={decorate}] (1,6.5) -- ++(-26.5:4.576) arc(150:-154.2:1) -- (1,1.5);
		\end{scope}
		\begin{scope}[thick, decoration={markings,
    			mark=at position 0.52 with {\arrow{>}}}]	
			\draw[postaction={decorate}] (1,6.5) -- ++(-26.5:4.576) arc(150:-154.2:1) -- (1,1.5);
		\end{scope}
		\begin{scope}[thick, decoration={markings,
    			mark=at position 0.9 with {\arrow{>}}}]	
			\draw[postaction={decorate}] (1,6.5) -- ++(-26.5:4.576) arc(150:-154.2:1) -- (1,1.5);
		\end{scope}
		\draw (6.7,3.25) node [right] {$\Gamma$};
		\draw (2,3.5) node [right] {$\Lambda$};
		\draw (9.05,4.3) node [left] {$\mathrm{Re}(z)$};
		\draw (6.05,6.7) node [left] {$\mathrm{Im}(z)$};
		\draw[dashed] (6,4)-- ++(25:1);
		\draw[dashed,->] (4.5,4) arc(180:153:1.5);
		\draw (6.7,4.45) node [left] {$\delta$};
		\draw (4.6,4.4) node [left] {$\theta$};
	\end{tikzpicture}

\end{center}
\end{Def}

\noindent
The operator $A_z$, $\Re(z)<0$, is well defined since, from Theorem \ref{sect}, $\|(A-\lambda I)^{-1}\|_{\mathcal{L}(L^2)}\leq \frac{C}{|\lambda|}$ and this gives the absolute convergence of the integral. The definition can be extended to arbitrary $z\in\C$:

\begin{Def}
\label{allz}
Let $A$ be a $SG$-operator that satisfies Assumptions \ref{assu}. Define
\[ 
A^z=
\begin{cases}
	A_z & \mbox{for $\Re(z)<0$}
	\\
	A_{z-l} A^l& \mbox{for $\Re(z)\ge0$, with $l = 1, 2,  \dots,\; \Re(z-l)<0$}.
\end{cases}
\]
\end{Def}
\begin{Prop}
\label{proexp}
\begin{itemize}
\item[i)] The Definition of $A^z$ for $\Re(z)\geq 0$ does not depend on the integer $l$.
\item[ii)] $A^z A^s= A^{z+s}$ for all $z, s \in \C$.
\item[iii)] $A^k= \underbrace{A \circ \ldots \circ A}_{k \mbox{ times}}$ when $z$ coincides with the positive integer $k$.
\item[iv)] If $A \in L^{m_1, m_2}$ then $A^z \in L^{m_1 z, m_2 z}$.
\end{itemize}
\end{Prop}
\noindent
The proof can be found, e.g., in  \cite{SC88} and \cite{SH87}. Note that the definition and properties of $SG$-symbols and operators with complex double order $(z_1,z_2)$ are analogous to those given above, with $\Re(z_1),\Re(z_2)$ in place of $m_1,m_2$, respectively.

\begin{rem}
\label{ellrem}
An application of Lemma \ref{elllemma} implies that the symbol of the operator $A^z$ has the form
\[
\simb(A^z)= \frac{1}{2\pi i} \int_{\partial^+ \Omega_{\x, \csi}} \lambda^z \, \simb((A-\lambda I)^{-1}) d\lambda.
\]
\end{rem}

\noindent
It is a fact that, given a $SG$-classical operator $A$ satisfying Assumptions \ref{assu}, $A^z$ is still classical: Maniccia, Schrohe and Seiler proved this in \cite{MSS06} by direct computation, finding the $SG$-classical expansion of $\simb(A^z)$. For future reference and sake of completeness, we prove here the same result by a different technique, which makes use of the identification between $SG$-classical symbols and $C^\infty(\mathbb{B}^n \times \mathbb{B}^n)$ given in Theorem \ref{thm4.6.1.1}.

\begin{Thm}
\label{sgrezneg}
Let $A \in L_{\cl}^{m_1,m_2}$ be an operator satisfying Assumptions \ref{assu}. Then $A^z$, $\Re (z) < 0$, is $SG$-classical of order $(m_1 z,m_2z)$.
\end{Thm}
\begin{proof}
In this proof we use vector notation for the orders, setting $\m=(m_1,m_2)$, $\textbf{e}=(1,1)$.
By Lemma \ref{elllemma} and Remark \ref{ellrem} we know that
\[
a^z=\simb(A^z)= \frac{1}{2 \pi i} \int_{\partial^+ \Omega_{\x, \csi}} \lambda^z \,\simb((A-\lambda I)^{-1}) d\lambda.
\]
We have to prove that $a^z\in SG^{\m z}_\cl$. To begin, we claim that
\[
b_{\m z}(x,\xi) = \frac{1}{2\pi i} \int_{\partial^+ \Omega_{\x, \csi}} 
\lambda^z(a_{\m}(x,\xi)-\lambda)^{-1} d\lambda = [a_{\m}(x,\xi)]^z \in SG^{\m z}_{\cl},\, \Re (z)  < 0.
\] 
In view of Theorem \ref{thm4.6.1.1}, it is enough to show that $(D^{\m z}b_{\m z})(y,\eta) \in C^{\infty}(\mathbb{B}^n \times \mathbb{B}^n)$. 
For $\mathbf{t}=(t_1,t_2)$, set $w_{\mathbf{t}}(y,\eta) = (1-[\eta])^{t_{1}} (1-[y])^{t_{2}}$. By the change of variable $\lambda = w_{-\m}(y,\eta) \mu$, we get
\begin{eqnarray*}
    (D^{\m z}b_{\m z})(y,\eta) & = & \frac{1}{2\pi i} \int_{\partial^+ \Omega_{\norm{\chi(y)}, \norm{\chi(\eta)}}} 
\frac{\lambda^z \, w_{\m z}(y,\eta)}{a_{\m}(\chi(y),\chi(\eta))-\lambda} d\lambda
\\
& = & \frac{1}{2\pi i} \int_{\partial^+ \widetilde{\Omega}_{y,\eta}} 
\frac{\mu^z \, w_{-\m}(y,\eta)}{a_{\m}(\chi(y),\chi(\eta))-\mu \, w_{-\m}(y,\eta)} 
d\mu
\\
& = & \frac{1}{2\pi i} \int_{\partial^+ \widetilde{\Omega}_{y,\eta}} 
\frac{\mu^z}{(D^\m a_{\m})(y,\eta)-\mu} d\mu
\end{eqnarray*}
By Lemma \ref{elllemma},
$|a_{\m}(x,\xi) - \lambda| \ge c(\norm{\xi}^{m_{1}}
    \norm{x}^{m_{2}} + |\lambda|)$, which implies
    $|(D^\m a_{\m})(y,\eta) - \mu| \ge c(1 + |\mu|)$, so that
$D^{\m z} b_{\m z} \in C^\infty(\mathbb{B}^n\times \mathbb{B}^n)$, as claimed.

\noindent
By the parametrix construction in the $SG$-calculus, and in view of the $\Lambda$-ellipticity of $A$, we have
\[
(A-\lambda I)^{-1} = \op((a-\lambda)^{-1})+\op(c)+\op(q),
\]
where $q\in SG^{-\infty}$, $c=\simb((A-\lambda I)^{-1}-\op((a-\lambda)^{-1}))\sim\sum_{j=1}^\infty c_j$, $c_j=r_j\,(a-\lambda)^{-1}$, $r_j\in 
SG^{-j\mathbf{e}}_\cl$, $j\ge 1$, see \cite{MSS06}. We can then write
\begin{equation}
\label{eq:ris}
\begin{split}
a^z=& \frac{1}{2\pi i}\int_{\partial^+ \Omega_{\x, \csi}} \lambda^z \simb((A-\lambda I)^{-1})d\lambda\\
\sim&\frac{1}{2\pi i}\int_{\partial^+ \Omega_{\x, \csi}} \lambda^z (a-\lambda)^{-1}d\lambda
+\frac{1}{2\pi i}\sum_{j=1}^{\infty}\int_{\partial^+ \Omega_{\x, \csi} }\lambda^z r_j(a-\lambda)^{-1}d\lambda\\
+&\frac{1}{2\pi i}\int_{\partial^+ \Omega_{\x, \csi}} \lambda^z q \,d\lambda.
\end{split}
\end{equation}
Let us consider the first term.
The operator $A$ is $SG$-classical so $a= a_\m + r$, $r \in SG^{\m-\mathbf{e}}_\cl$.
We have, for all $N \in \N{}$,
\[
	\begin{split}
    (a-\lambda)^{-1}&= (a_{\m} + r - \lambda)^{-1} \\
    &= (a_{\m}-\lambda)^{-1}(1+(a_{\m}-\lambda)^{-1}r)^{-1}\\
    &= (a_{\m}-\lambda)^{-1}\left[ \sum_{k=0}^N
    (-1)^k(a_{\m}-\lambda)^{-k} r^k \right.\\
    &\hspace{2cm}\left.\phantom{\sum_{k=0}^N}+ (-1)^{N+1}(1+(a_{\m}-\lambda)^{-1}r)^{-1}
    (a_{\m}-\lambda)^{-(N+1)}r^{N+1}\right],
    \end{split}
\]
and then
\begin{eqnarray*}
    b& = &
       \frac{1}{2\pi i} \int_{\partial \Omega_{\x, \csi}} 
       \lambda^z(a-\lambda)^{-1}
       d\lambda
\\
& = & \underbrace{\frac{1}{2\pi i} \int_{\partial \Omega_{\x, \csi}} 
       \lambda^z(a_{\m}-\lambda)^{-1}
       d\lambda}_{b_{\m z}}
       + \sum_{k=1}^{N}
       \underbrace{\frac{(-1)^k}{2\pi i} \int_{\partial \Omega_{\x, \csi}} 
       \lambda^z r^k(a_{\m}-\lambda)^{-k-1}
       d\lambda}_{b_k} + R_{N},
\end{eqnarray*}
where 
\[
R_N= \frac{(-1)^{N+1}}{2\pi i} \int_{\partial \Omega_{\x, \csi}} {\lambda^z}(1+(a_{\m}-\lambda)^{-1}r)^{-1}
    (a_{\m}-\lambda)^{-(N+2)}r^{N+1} d\lambda.
\]
By the calculus and the hypotheses, it turns out that $R_{N} \in SG^{\m-(N+1)\mathbf{e}}$. 
Moreover, $b_{k} \in SG^{\m z-k\mathbf{e}}_{\cl}$, $k\ge 1$. Indeed, as above,
\begin{equation}
\begin{split}
\label{eq:com}
    (D^{\m z-k \mathbf{e}}&b_{k})(y,\eta)  = \frac{(-1)^k}{2\pi i} \int_{\partial^+ \Omega_{\norm{\chi(y)}, \norm{\chi(\eta)}}} 
\frac{\lambda^z \, w_{\m z-k\mathbf{e}}(y,\eta)\, r^k(\chi(y),\chi(\eta))}
     {(a_{\m}(\chi(y),\chi(\eta))-\lambda)^{k+1}} d\lambda
\\
&= \frac{(-1)^k}{2\pi i} \int_{\partial^+ \widetilde{\Omega}_{y,\eta}} 
\frac{\mu^z \, w_{-\m z}(y,\eta) \, w_{\m z-k\mathbf{e}}(y,\eta)\,  
      r^k(\chi(y),\chi(\eta))}
     {w_{\m}(y,\eta) \, (a_{\m}(\chi(y),\chi(\eta))-\mu \, w_{-\m}(y,\eta))^{k+1}}
     d\mu
     \\
&= \frac{(-1)^k}{2\pi i} \int_{\partial^+ \widetilde{\Omega}_{y,\eta}}
\frac{\mu^z \, ( (D^{\m-\mathbf{e}} r)(y,\eta) )^k}
     {((D^\m a_{\m})(y,\eta)-\mu)^{k+1}} d\mu \in 
     C^\infty(\mathbb{B}^n\times \mathbb{B}^n).
\end{split}
\end{equation}
Theorem \ref{thm4.6} then gives $b\in SG^{\m z}_\cl$ with $b-b_{\m z}\in SG^{\m z-\mathbf{e}}_\cl$. 
In a completely similar fashion, it is possible to prove that the asymptotic sum in \eqref{eq:ris}
gives a symbol in $SG^{\m z-\mathbf{e}}_\cl$, since $D^{-j\mathbf{e}}r_j$ is smooth and uniformly bounded,
together with its derivatives, with respect to $\mu$ (see \cite{MSS06} for more details).
Finally, it is easy to see that the third term in \eqref{eq:ris} gives a smoothing operator.
Again by Theorem \ref{thm4.6}, $a^z \in SG^{\m z}_{\cl}$, with $a^z=[a_{\m}(x,\xi)]^z \mod SG^{\m z-\mathbf{e}}_\cl$.
\end{proof}

\begin{rem}
\label{alllz}
By Definition \ref{allz},
\[
A^z= A^l \circ A^{z-l}, \quad \Re(z-l)<0,
\]
and, by Theorem \ref{sgrezneg}, we obtain that $A^z$ is a $SG$-classical operator for all $z \in \C$. 
So, denoting $a^l_{m_1 l-j, \cdot}(x, \xi), j=0,1, \ldots$, the terms of the homogeneous expansion with respect to $\xi$ of $A^l$, the 
$SG$-calculus implies
\begin{equation}
\label{svxi}
a^z_{m_1z-j,\cdot}(x, \xi)= \frac{1}{\alpha!}\sum_{|\alpha|+i+k=j} \partial^{\alpha}_\xi a^l_{m_1 l-i, \cdot} D^\alpha_x a^{z-l}_{m_1(z-l)-k, \cdot}.
\end{equation}
The same holds for the $x$-expansion
\begin{equation}
\label{svx}
a^z_{\cdot, m_2z-k}= \frac{1}{\alpha!}\sum_{|\alpha|+i+j=k} \partial^{\alpha}_\xi a^l_{\cdot, m_2 l-i} D^\alpha_x a^{z-l}_{\cdot,m_2(z-l)-j}.
\end{equation}
\end{rem}

\noindent
The following Proposition is immediate, in view of the proof of Theorem \ref{sgrezneg}:
\begin{Prop}
	\label{invprinc}
	The top order terms in the expansions \eqref{svxi}, \eqref{svx} are such that
	\begin{equation}
	\label{invprin}
		\begin{split}
			a^z_{m_1z,\cdot}&= (a_{m_1,\cdot})^z,\\
			a^z_{\cdot,m_2z}&=(a_{\cdot,m_2})^z,\\
			a^z_{m_1z,m_2z}&=(a_{m_1,m_2})^z.
		\end{split}
		\end{equation}
\end{Prop}

\begin{rem}
In order to define $A^z$ we do not need $m_1, m_2$ integer numbers. Anyway, this hypothesis is essential in the definition of Wodzicki residue given below, so we included it from the very beginning in Assumptions \ref{assu}. 
\end{rem}

%
%

\section{\texorpdfstring{$\zeta$ function and Wodzicki residue for $SG$-classical operators on $\R^n$}
{Zeta function of the operator A}}
\label{seczeta}
\setcounter{equation}{0}

In \cite{SC88} E. Schrohe noticed that, for $A\in L^{m_1,m_2}$ such that $\Re(z)m_1<-n$ and $\Re(z)m_2<-n$, $A^z$ is trace class, so he defined
\begin{equation}
\label{zeta}
\zeta(A, z)= Sp(A^z)=\int K_{A^z}(x, x) dx,
\end{equation}
where $Sp$ is the spur of $A^z$, i.e.,  a trace on the algebra of trace class operators. 
Assuming that $A$ is $SG$-classical and elliptic, we want to study the meromorphic extension of $\zeta(A,z)$: this will allow to define
trace operators, in connection with the residues of $\zeta(A,z)$.
We first consider the kernel $K_{A^z}(x, y)$ of
the operator $A^z$ defined in \ref{sec:zpower}.
The information provided by the knowledge of the homogeneous expansions of the symbol of $A^z$ allows to investigate in detail the properties of $K_{A^z}(x,y)$ on the diagonal $(x,x)$. 
\begin{Thm}
\label{kernel}
Let $A$ be an elliptic operator that satisfies Assumptions \ref{assu}. Then,
$K_{A^z}(x,x)$ is a holomorphic function for $\Re(z)<-\frac{n}{m_1}$ and admits, at most, simple
poles at the points $z_j=\frac{j-n}{m_1}$, $j=0, 1, \ldots$.
\end{Thm} 
\begin{proof}
Let us consider the kernel $K_{A^z}(x, y)$ on the diagonal $(x,x)$, given by
\[
\begin{split}
K_{A^z}(x,x)&=\frac{1}{(2\pi)^n} \int_{\R^n} \simb(A^z)(x, \xi) d\xi\\
&= \frac{1}{(2\pi)^n}\int_{|\xi|<1}a^z(x, \xi) d\xi+ \frac{1}{(2\pi)^n}\int_{|\xi|\geq 1} a^z(x, \xi) d\xi.
\end{split}
\]
Clearly, the first integral converges, and the resulting function is holomorphic, so we can focus on
\[
\begin{split}
\int_{|\xi|\geq 1} a^z(x, \xi)d\xi=& \int_{|\xi|\geq 1}\sum_{j=0}^{p-1} a^z_{m_1z-j, \cdot}\left(x, \frac{\xi}{|\xi|}\right) |\xi|^{m_1z-j} d\xi\\
&+\int_{|\xi|\geq 1} r_{m_1 z-p,\cdot} (x, \xi) d\xi.
\end{split}
\]
The number $p$  can be chosen such that $m_1 \Re(z)-p<-n$: this means that we have to deal with  the terms appearing in the sum for $j=0,\ldots, p-1$. Switching to polar coordinates
$\xi=\rho \omega$, $\rho \in [1, \infty)$, $\omega \in \mathbb{S}^{n-1}$,
\[
\begin{split}
\int_{|\xi|\geq 1} a^z(x, \xi)d\xi&=\sum_{j=0}^{p-1} \int_1^{\infty} \rho^{m_1z-j+n-1}d\rho \int_{\mathbb{S}^{n-1}} a^z_{m_1z-j}(x, \theta) d\theta \\
&+\int_{|\xi|\geq 1} r_{m_1 z-p,\cdot} (x, \xi) d\xi.
\end{split}
\]
To have convergence in the first integral, we must impose $m_1\Re (z)+n<0$, i.e., $\Re(z)<-\frac{n}{m_1}$. Evaluating the integral, we find
\[
\begin{split}
\int_{|\xi|\geq 1} a^z(x, \xi)d\xi&= -\sum_{j=0}^{p-1}\frac{1}{m_1z-j+n} \int_{\mathbb{S}^{n-1}} a^z_{m_1z-j, \cdot}(x, \theta) d\theta \\
&+ \int_{|\xi|\ge1} r^z_{m_1z-p}(x, \xi) d\xi.
\end{split}
\]  
This proves that $K_{A^z}(x,x)$ is holomorphic for $\Re(z)<-\frac{n}{m_1}$, and that it can be extended as a meromorphic function on the whole complex plane with, at most, simple poles at the points $z_j=\frac{j-n}{m_1}$, $j=0, 1, \ldots$
\end{proof}
\begin{rem}
As in the case of a compact manifold, see \cite{SE67}, we can prove that the kernel $K_{A^z}(x,x)$ is regular for $z=0$ and, if $A$ is a differential operator, $K_{A^z}(x,x)$ is also regular for all  integer.
\end{rem}
\noindent Now we proceed to examine the properties of $\zeta(A,z)$:
\begin{Thm}
\label{fzeta}
Let $A$ be an elliptic operator that satisfies Assumptions \ref{assu}, and define
\begin{equation}
\label{fzetaeq}
\zeta(A, z)= \int_{\R^n} K_{A^z}(x, x) dx= \frac{1}{(2\pi)^n}\int_{\R^n}\int_{\R^n} \simb(A^z)(x,\xi)  d\xi dx.
\end{equation}
The function $\zeta(A,z)$ is holomorphic for $\Re(z)<\min \{-\frac{n}{m_1}, -\frac{n}{m_2}\}$.
Moreover, it can be extended as a meromorphic function with, at most, poles at the points
\[
z_j^1=\frac{j-n}{m_1}, \, j=0, 1, \ldots, \quad z^2_k=\frac{k-n}{m_2}, \, k=0, 1, \ldots
\]
Such poles can be of order two if and only if there exist integers $j, k$ such that
\begin{equation}
\label{conpoli}
z_j^1=\frac{j-n}{m_1}=\frac{k-n}{m_2}=z_k^2.
\end{equation}

\end{Thm} 
\begin{proof}
We divide $\R^{2n}$ into the four regions
\[
\begin{split}
\{(x, \xi)\colon |x|\leq1, |\xi|\leq1\},\quad \{(x, \xi)\colon |x|<1, |\xi|> 1\},\\
 \{(x, \xi)\colon |x|>1, |\xi|<1\},\quad \{(x, \xi)\colon |x|\geq1, |\xi|\geq 1\}.
\end{split}
\]
Setting, as above, $a^z=\simb(A^z)$, we can write
\[
\begin{split}
\zeta_1(A,z)&= \frac{1}{(2\pi)^n}\int_{|x|\leq 1} \int_{|\xi|\leq1} a^z(x, \xi) d\xi dx,\\
\zeta_2(A, z)&= \frac{1}{(2\pi)^n}\int_{|x|<1}\int_{|\xi|>1} a^z(x, \xi) d\xi dx,\\
\zeta_3(A,z)&= \frac{1}{(2\pi)^n}\int_{|x|>1}\int_{|\xi|<1} a^z(x, \xi) d\xi dx,\\
\zeta_4(A, z)&=\frac{1}{(2\pi)^n}\int_{|x|\geq 1}\int_{|\xi|\geq 1}a^z(x, \xi) d\xi dx,\\
\zeta(A,z)&=  \sum_{i=1}^4 \zeta_i(A,z),
\end{split}
\]
and examine each term of the sum separately.
\begin{itemize}
\item [1)]The analysis of this term is straightforward. Since we integrate $a^z$, holomorphic function in $z$ and smooth with respect to $(x,\xi)$, on a bounded set with respect to $(x,\xi)$,  $\zeta_1(A,z)$ is holomorphic.
\item [2)]
Using the asymptotic expansion of $a^z$ with respect to $\xi$,  we can write
\[
\begin{split}
\zeta_2(A,z)=&-\frac{1}{(2\pi)^n} \sum_{j=0}^{p-1}\frac{1}{m_1z+n-j} \int_{|x|<1}\int_{\mathbb{S}^{n-1}} a^z_{m_1z-j, \cdot} (x, \theta) d\theta dx\\
&+\frac{1}{(2\pi)^n}\int_{|x|<1}\int_{\R^n} r^z_{m_1z-p, \cdot}(x, \xi) d\xi dx.
\end{split}
\]
Choosing $p>m_1 \Re (z)+n$, the last integral is convergent. For the sum, we can argue as in the  proof of the Theorem \ref{kernel}. So $\zeta_2(A, z)$ is holomorphic for $\Re(z)<-\frac{n}{m_1}$ and has, at most, poles at the points $z^1_j=\frac{j-n}{m_1}$.
\item[3)]

To discuss this term we need the asymptotic expansion of $a^z$ with respect to $x$. Using Proposition \ref{allz}, we can write
\[
a^z(x, \xi)=\sum_{k=0}^{q-1} a^z_{\cdot,m_2z-k} (x, \xi)+ t^z_{\cdot, m_2z-q}(x, \xi),
\]
which implies
\[
\begin{split}
\zeta_3(A,z)&= \frac{1}{(2\pi)^n}\int_{|x|>1}\int_{|\xi|<1} \sum_{k=0}^{q-1} a^z_{\cdot,m_2z-k} \left(\frac{x}{|x|}, \xi\right) |x|^{m_2z-k}d\xi dx\\
&+ \frac{1}{(2\pi)^n}\int_{|x|>1}\int_{|\xi|<1} t^z_{\cdot, m_2z-q}(x, \xi) d\xi dx.
\end{split}
\]
Now, switching to polar coordinates, we can write
\[
\begin{split}
\zeta_3(A,z)&= \frac{1}{(2\pi)^n}\sum_{k=0}^{q-1} \int_{1}^{\infty} \rho^{m_2z+n-1-k} \int_{\mathbb{S}^{n-1}} \int_{|\xi|<1} a^z_{\cdot, m_2z-k}(\theta, \xi)d\xi d\theta d\rho\\
&+\frac{1}{(2\pi)^n}\int_{|x|>1}\int_{|\xi|<1} t^z_{\cdot, m_2z-q}(x,\xi) d\xi dx.
\end{split}
\]
Arguing as in point (2), it turns out that $\zeta_3(A,z)$ is holomorphic for $\Re(z)<-\frac{n}{m_2}$ and can be extended as a meromorphic function on the whole complex plane with, at most, poles at the points $z^2_k= \frac{k-n}{m_2}$.

\item[4)] To treat the last term we need to use both the expansions with respect to $x$ and with respect to $\xi$.
We first expand $a^z$ with respect to $\xi$
\[
\begin{split}
\zeta_4(A,z)&=\frac{1}{(2\pi)^n}\sum_{j=0}^{p-1}\int_{|x|\geq 1}\int_{|\xi|\geq 1}  a^z_{m_1z-j, \cdot}(x, \xi)d\xi dx\\
&+\frac{1}{(2\pi)^n} \int_{|x|\geq 1}\int_{|\xi|\geq 1} r^z_{m_1z-j, \cdot}(x, \xi)d\xi dx.
\end{split}
\]
In order to integrate over $|\xi|\geq 1$, we assume $\Re(z)< -\frac{n}{m_1}$. Now, switching to polar coordinates and integrating the radial part, we can write
\[ 
\begin{split}
\zeta_4(A,z)=& -\frac{1}{(2\pi)^n}\sum_{j=0}^{p-1} \int_{|x|\geq 1} \frac{1}{m_1z+n-j}\int_{\mathbb{S}^{n-1}}a^z_{m_1z-j, \cdot}(x, \theta)d\theta dx\\
&+ \frac{1}{(2\pi)^n}\int_{|x|\geq 1}\int_{|\xi|\geq 1} r^z_{m_1z-p, \cdot}(x, \xi)d\xi dx.
\end{split}
\]
Now, in order to integrate over $|x|\geq 1$, we expand with respect to $x$ 
\[
\begin{split}
\zeta_4(A,z)=& -\frac{1}{(2\pi)^n}\sum_{k=0}^{q-1} \sum_{j=0}^{p-1} \int_{|x|\geq 1} \frac{1}{m_1z+n-j} \int_{\mathbb{S}^{n-1}} a^z_{m_1z-j, m_2z-k}(x, \theta)d\theta dx\\
&-\frac{1}{(2\pi)^n}\sum_{j=0}^{p-1}\frac{1}{m_1z+n-j} \int_{|x|\geq 1}\int_{\mathbb{S}^{n-1}}t^z_{m_1z-j, m_2z-q}(x, \theta) dx d\theta\\
&+\frac{1}{(2\pi)^n}\sum_{k=0}^{q-1} \int_{|x|\geq 1} \int_{|\xi|\geq 1} r^z_{m_1z-p, m_2z-k}(x,\xi) d\xi d\theta\\
&+\frac{1}{(2\pi)^n}\int_{|x|\geq 1}\int_{|\xi|\geq 1} r^z_{m_1z-p, m_2z-q}(x, \xi)dx d\xi.
\end{split}
\]
Imposing $\Re(z)<-\frac{n}{m_2}$, and integrating the radial part with respect to the $x$, we obtain
\[
\begin{split}
\zeta_4(A,z)&= \frac{1}{(2\pi)^n} \sum_{k=0}^{q-1} \sum_{j=0}^{p-1} \frac{1}{m_2z+n-k} \frac{1}{m_1z+n-j} I_{m_1z-j}^{m_2z-k}\\
	&-\frac{1}{(2\pi)^n}\sum_{j=0}^{p-1}\frac{1}{m_1z+n-j}R_{j,q}(z)\\
	&-\frac{1}{(2\pi)^n}\sum_{k=0}^{q-1} \frac{1}{m_2z+n-k} R_{p,k}(z) + R_{p,q}(z)
\end{split}
\]
where 
\begin{equation}
\label{Iint}
I_{m_1z-j}^{ m_2z-k}=\int_{\mathbb{S}^{n-1}}\int_{\mathbb{S}^{n-1}} a^z_{m_1z-j, m_2z-k}(\theta', \theta)d\theta d\theta'.
\end{equation}
$R_{j,k}, R_{p,k}, R_{p,q}$ are holomorphic for $\Re (z) < \min\{-\frac{n}{m_1}, -\frac{n}{m_2}\}$, since $p, q$ are arbitrary. So 
we obtain that $\zeta_4(A,z)$ is holomorphic for $\Re(z)< \min \{-\frac{n}{m_1} , -\frac{n}{m_2}\}$ and can be extended as a meromorphic function
on the whole complex plane with, at most, poles at the points $z^1_j=\frac{j-n}{m_1}, z^2_k= \frac{k-n}{m_2}$. Clearly these poles can be of order two when the conditions \eqref{conpoli} in the statement are fulfilled.
\end{itemize}
The proof  is complete.
\end{proof}

We can now prove two Theorems which show the relation between $\zeta(A,z)$ and the functionals introduced by F. Nicola \cite{NI03}, namely

\begin{equation}
\label{trnicola}
\begin{split}
\textnormal{Tr}_{\psi, e}(A)&= \frac{1}{(2 \pi)^{n}} \int_{\mathbb{S}^{n-1}}\int_{\mathbb{S}^{n-1}} a_{-n, -n}(\theta, \theta') d\theta' d\theta=
\frac{1}{(2 \pi)^{n}} I^{-n}_{-n},\\
\widehat{\textnormal{Tr}}_\psi(A)&= \frac{1}{(2 \pi)^{n}} \lim_{\tau \to \infty} \Big[ \int_{|x|\leq \tau}\int_{\mathbb{S}^{n-1}}a_{-n, \cdot}(x, \theta)d\theta dx\\
&\hspace{2.5cm}-(\log\tau) \, I^{-n}_{-n} - \sum_{k=0}^{m_2+n-1} \frac{\tau^{m_2-k}}{(m_2-k)}I_{-n}^{m_2-k} \Big],\\
\widehat{\textnormal{Tr}}_e(A)&= \frac{1}{(2 \pi)^{n}} \lim_{\tau \to \infty} \Big[\int_{\mathbb{S}^{n-1}} \int_{|\xi|\leq \tau}a_{ \cdot,-n}(\theta, \xi) d\xi d\theta\\
&\hspace{2.5cm}-(\log\tau)\, I^{-n}_{-n} - \sum_{j=0}^{m_1+n-1}\frac{\tau^{m_1-j}}{(m_1-j)}I_{m_1-j}^{-n}\Big],
\end{split}
\end{equation}
where $I_{m_1-j}^{-n}, I_{-n}^{m_2-k}$ are integrals of the form \eqref{Iint} with $a_{m_1-j,-n}$ and  $a_{-n,m_2-k}$ in place of $a^z_{m_1z-j,m_2z-k}$,
respectively.
We define the following new functional, that we call the \textit{angular term}
\begin{equation}
\label{trang}
\widehat{\TR}_{\theta}(A)= \frac{1}{ (2 \pi)^{n}} \int_{\mathbb{S}^{n-1}}\int_{\mathbb{S}^{n-1}} \left.\frac{d}{dz}(a^z_{m_1z-n-m_1, m_2z-n-m_2})\right|_{z=1} (\theta,\theta')d\theta' d\theta.
\end{equation}
\begin{rem}
In general, it is rather cumbersome to evaluate the angular term defined in \eqref{trang}. In the case $m_1=m_2=-n$, the computation is easier: by Proposition \ref{invprinc},
\[
\left.\frac {d}{dz} (a^z_{-n \,z, -n \,z}) \right|_{z=1}=\lim_{z \to 1}\frac{a^z_{-n \, z, -n \, z}- a_{-n,-n}}{z-1}= a_{-n,-n} \cdot \log(a_{-n,-n}) .
\]
\end{rem}
\begin{Thm}
\label{thm:equiv}
Let $A$ be an operator satisfying Assumptions \ref{assu}. Then, defining
\begin{equation}
\label{trnic}
\begin{split}
\TR(A)&=m_1m_2\Res^2_{z=1}(\zeta(A,z))=m_1m_2\lim_{z \to 1} (z-1)^2 \zeta(A,z),
\end{split}
\end{equation}
we have
\begin{equation}
\label{eqnic}
\TR(A)=\textnormal{Tr}_{\psi, e}(A).
\end{equation}
\end{Thm}

\begin{proof}
To evaluate the limit we split again $\zeta(A,z)$ into the four terms already examined in the proof of Theorem \ref{fzeta}.  We get:
\begin{itemize}
\item[1)] $\displaystyle\lim_{z\to1}(z-1)^2 \zeta_1(A,z)=0$, since $\zeta_{1}(A,z)$ is holomorphic;
\item[2)] $\displaystyle\lim_{z\to1}(z-1)^2 \zeta_2(A,z)=0$, since $\zeta_2(A,z)$ has a pole of order one at $z=1$;
\item[3)] Similarly,  $\displaystyle\lim_{z \to 1} (z-1)^2 \zeta_3(A,z)=0$;
\item[4)] Finally,
\[
\lim_{z\to 1} (z-1)^2 \zeta_4(A,z)= \frac{1}{ m_1 m_2 (2\pi)^n} \int_{\mathbb{S}^{n-1}}\int_{\mathbb{S}^{n-1}}a^1_{-n, -n} (\theta, \theta') d\theta' d\theta.
\] 
Now the theorem follows from Proposition \ref{proexp}, which gives $A^1= A$, so that $a^1_{-n, -n}=a_{-n,-n}$.
\end{itemize}
\end{proof}

\begin{Thm}
\label{trxxi}
Let $A$ be an operator that satisfies Assumptions \ref{assu}. Then, defining
\begin{equation}
\label{xxi}
\widehat{\TR}_{x,\xi}(A)= \lim_{z\to 1} (z-1) \left[\zeta(A, z)- \frac{\Res^2_{z=1}(\zeta(A,z))}{(z-1)^2}\right],
\end{equation}
we have
\begin{equation}
\label{eq:reltr}
\widehat{\TR}_{x,\xi}(A)= -  \frac{1}{m_1}\widehat{\textnormal{Tr}}_\psi(A) - \frac{1}{m_2} \widehat{\textnormal{Tr}}_e(A) + \frac{1}{m_1 m_2}\widehat{\TR}_{\theta}(A).
\end{equation}
\end{Thm}
\begin{proof}
We notice that the function 
\[
\zeta(A,z)- \frac{\Res^2_{z=1}(\zeta(A,z))}{(z-1)^2}
\]
is meromorphic with a simple pole at the point $z=1$, so the limit \eqref{xxi} exists and is finite.
In order to prove the assertion, we use a decomposition of $\R^{2n}$ into four sets defined by means of a parameter $\tau>1$,
\[
\begin{split}
&D_1=\{(x, \xi)\colon |x|\leq \tau, |\xi|\leq \tau\}, \quad D_2=\{(x, \xi)\colon |x|<\tau, |\xi|> \tau\},\\
&D_3=\{(x, \xi)\colon |x|>\tau, |\xi|<\tau\},\quad D_4=\{(x, \xi)\colon |x|\geq \tau, |\xi|\geq \tau\}.
\end{split}
\]
and set 
\[
\begin{split}
&\zeta_i(A,z)=\iint_{D_i} a^z(x,\xi)d\xi dx, \quad i=1,\ldots, 4.
\end{split}
\]

\begin{itemize}
\item[1)] 
$D_1$ is a compact set: $\zeta_1(A,z)$ is then holomorphic, so that, for any $\tau\ge 1$, $\displaystyle L_1=\lim_{z\to 1}(z-1)\zeta_1(A,z)=0$.
\item[2)] Expanding $a^z$ with respect to $\xi$, we find
\[
\begin{split}
\zeta_2(A,z)=&-\frac{1}{(2\pi)^n}\int_{|x|<\tau} \sum_{j=0}^{p-1}\frac{\tau^{m_1z+n-j}}{m_1z+n-j} \int_{\mathbb{S}^{n-1}} a^z_{m_1z-j, \cdot} (x, \theta) d\theta dx\\
&+\frac{1}{(2\pi)^n}\int_{|x|<\tau}\int_{\R^n} r_{m_1z-p, \cdot}(x, \xi) d\xi dx.
\end{split}
\]
For $p$ big enough, $r_{m_1z-p, \cdot}$ is absolutely integrable with respect to $\xi$,
So we have, for $q$ big enough, and any $\tau\ge1$,
\[
L_2=\lim_{z\to 1}(z-1)\zeta_2(A,z)=- \frac{1}{m_1(2\pi)^n}\int_{|x|< \tau} \int_{\mathbb{S}^{n-1}} a^1_{-n,\cdot}(x, \theta) d\theta dx,
\]
since any term in the limit goes to zero, apart the one corresponding to $j=n+m_1$.
\item[3)] Similarly, using the expansion of $a^z$ with respect to $x$,
\[
\begin{split}
\zeta_3(A,z)&= - \frac{1}{(2\pi)^n}\sum_{k=0}^{q-1} \frac{\tau^{m_2z+n-k}}{m_2z+n-k} \int_{\mathbb{S}^{n-1}} \int_{|\xi|<\tau} a^z_{\cdot, m_2z-k}(\theta, \xi) d\xi d\theta \\
&+ \frac{1}{(2\pi)^n}\int_{x\geq \tau} \int_{\xi \leq \tau} t^z_{\cdot, m_2z-q}(x, \xi) d\xi dx,
\end{split}
\]
so that, for $q$ big enough and any $\tau\ge1$,
\[
L_3=\lim_{z\to 1}(z-1)\zeta_3(A,z)= - \frac{1}{m_2(2\pi)^n}\int_{\mathbb{S}^{n-1}}\int_{|\xi| \leq \tau} a^1_{\cdot, -n}(\theta, \xi)d\xi d\theta.
\]
\item[4)] Expanding with respect to both the variables $x$ and $\xi$,
\[
\begin{split}
\zeta_4(A,z)&= \frac{1}{(2\pi)^n}\sum_{k=0}^{q-1} \sum_{j=0}^{p-1}\frac{\tau^{m_1z+n-j}}{m_1z+n-j}\frac{\tau^{m_2z+n-k}}{m_2z+n-k} I_{m_1z-j}^{m_2z-k} \\
&-\frac{1}{(2\pi)^n}\sum_{j=0}^{p-1}\frac{\tau^{m_1z+n-j}}{m_1z+n-j} \int_{|x|\geq \tau}\int_{\mathbb{S}^{n-1}}t^z_{m_1z-j, m_2z-q}(x, \theta) d\theta dx\\
&-\frac{1}{(2\pi)^n}\sum_{k=0}^{q-1} \frac{\tau^{m_2z+n-k}}{m_2z+n-k}\int_{\mathbb{S}^{n-1}} \int_{|\xi|\geq \tau} r^z_{m_1z-p, m_2z-k}(\theta,\xi) d\xi d\theta\\
&+\frac{1}{(2\pi)^n}\int_{|x|\geq \tau}\int_{|\xi|\geq \tau} r^z_{m_1z-p, m_2z-q}(x, \xi)dx d\xi
\end{split}
\]
So, for $p$ and $q$ big enough, and any $\tau\ge1$, we have
\[
\begin{split}
L_4=\lim_{z\to1}&(z-1)\left(\zeta_4(A,z)- \frac{\Res^2_{z=1}(\zeta(A,z))}{(z-1)^2}\right)=\\
&\,\phantom{+}\,\lim_{z\to 1} \frac{(z-1)}{m_1m_2(2 \pi)^{n}} \frac{\tau^{(m_1+m_2)(z-1)}-1}{(z-1)^2} I_{m_1z-n-m_1}^{m_2z-n-m_2}\\
&+\lim_{z\to 1} \frac{(z-1)}{m_1m_2(2 \pi)^{n}} \frac{I_{m_1z-n-m_1}^{m_2z-n-m_2} - I_{-n}^{-n}}
{(z-1)^2}\\
&+\frac{1}{m_2(2\pi)^n}\sum_{j=0, j\not=m_1+n}^{p-1} \frac{\tau^{m_1+n-j}}{m_1+n-j} I_{m_1-j}^{-n}\\
&+\frac{1}{m_1(2\pi)^n}\sum_{k=0, k\not=m_2+k}^{q-1}\frac{\tau^{m_2+n-k}}{m_2+n-k} I_{-n}^{m_2-k}\\
&-\frac{1}{m_1(2\pi)^n}\int_{|x|\geq \tau} \int_{\mathbb{S}^{n-1}}t^1_{-n, m_2-q}(x, \theta)d\theta dx\\
&-\frac{1}{m_2(2\pi)^n}\int_{\mathbb{S}^{n-1}} \int_{|\xi|\geq \tau} r^1_{m_1-p, -n}(\theta,\xi) d\xi d\theta\\
\end{split}
\]
The coefficients $I^{-n}_{m_1-j}$, $I^{m_2-k}_{-n}$, limits of corresponding integrals of the form \eqref{Iint}, are as in \eqref{trnicola}, while
the second limit coincides with the angular term $\widehat{\TR}_\theta(A)$, defined in \eqref{trang}. Moreover, the first limit goes to
\begin{align*}
\frac{1}{(2\pi)^{n}}I^{-n}_{-n}\,&\lim_{z\to 1}\frac{\tau^{(m_1+m_2)(z-1)}-1}{m_1m_2(z-1)}= 
\\
&=\frac{1}{(2\pi)^{n}}I^{-n}_{-n}\frac{m_1+m_2}{m_1m_2}\log \tau=
\frac{1}{(2\pi)^{n}}I^{-n}_{-n}\left(\frac{1}{m_2}\log \tau+ \frac{1}{m_1}\log \tau\right).
\end{align*}
\end{itemize}

\noindent
Clearly, $\widehat{\TR}_{x,\xi}(A)=\displaystyle\lim_{\tau\to+\infty}(L_1+L_2+L_3+L_4)=\lim_{\tau\to+\infty}(L_2+L_3+L_4)$. 
The two terms
\[
\int_{|x|\geq \tau} \int_{\mathbb{S}^{n-1}}t^1_{-n, m_2-q}(x, \theta)d\theta dx \,\mbox{ and }
\int_{\mathbb{S}^{n-1}}\int_{|\xi|\geq \tau} r^1_{m_1-p, -n} (\theta, \xi)d\xi d\theta
\]
in $L_4$ vanish for $\tau\to+\infty$, by the uniform continuity of the integral. Moreover, the terms in $L_4$ involving $I_{m_1-j}^{-n}$ and
$I^{m_2-k}_{-n}$ are relevant only for $m_1+n-j>0$ and $m_2+n-k>0$, respectively. Then, finally,
\[
\begin{split}
&\lim_{z\to1} (z-1)\left[\zeta(A,z)- \frac{\Res^2_{z=1}(\zeta(A,z))}{(z-1)^2}\right]=\\
&= \frac{1}{(2\pi)^n}\lim_{\tau \to \infty}\left[-\frac{1}{m_1}\int_{|x|\leq \tau}\int_{\mathbb{S}^{n-1}}a_{-n, \cdot}(x, \theta)dx d\theta\right.\\
&+\frac{1}{m_1}\sum_{k=0}^{m_2+n-1} \frac{\tau^{m_2-k}}{m_2-k} I_{-n}^{m_2-k}+ \frac{1}{m_1}(\log \tau) I^{-n}_{-n} \\
&-\frac{1}{m_2}\int_{\mathbb{S}^{n-1}}\int_{|\xi|\leq \tau} a_{\cdot, -n}(x, \theta)dx d\theta\\
&+\frac{1}{m_2}\sum_{j=0}^{m_1+n-1} \frac{\tau^{m_1-j}}{m_1-j}I_{m_1-j}^{-n}+ \left.\frac{1}{m_2}(\log \tau) I^{-n}_{-n}\right]\\
&+\frac{1}{m_1 m_2} \int_{\mathbb{S}^{n-1}}\int_{\mathbb{S}^{n-1}} \left.\frac{d}{dz}(a^z_{m_1z-n-m_1, m_2z-n-m_2})\right|_{z=1} (\theta, \theta')d\theta' d\theta.
\end{split}
\]
which, by \eqref{trnicola}, coincides with \eqref{eq:reltr}. The proof is complete.
\end{proof}

The functional $\TR$ can be extended to all $SG$-classical operators with integer order in a standard way, cfr. \cite{KA89}. Explicitely, let
$A\in L^{m_1, m_2}_\cl$, $m_1,m_2$ integers, and choose an elliptic operator $B$ of order $(m_1^\prime, m_2^\prime)$, satisfying Assumptions \ref{assu} and
$m_1^\prime>m_1$, $m_2^\prime>m_2$. We can define $\zeta(B+sA,z)$, $s\in(-1,1)$, and then set
\begin{equation}
\label{extres}
\TR(A)=m_1^\prime m_2^\prime \left.\frac{d}{ds} \Res^2_{z=1}(\zeta(B+ s A,z))\right|_{s=0}
\end{equation}
Using the expression of $\TR$ given in Theorem \ref{thm:equiv}, it is possible to prove that these definitions do not depend on the operator $B$. Moreover, with this approach it is also possibile to prove that $\TR$ is a trace on the algebra $\mathscr{A}$ of all
$SG$-classical operators with integer order modulo operators in $L^{-\infty}$, see \cite{KA89}.

%
\section{Wodzicki Residue and Weyl formula for $SG$-classical operators on manifolds with ends}
\label{cylexit}
\setcounter{equation}{0}

On the basis of the results of Section \ref{seczeta}, we can generalise the definition of Wodzicki Residue to all $SG$-classical operators with integer order on a manifold with ends $M$. In addition to this, the knowledge of the zeta function and of the related trace operators allows to describe the asymptotic behaviour of the counting function of elliptic positive selfadjoint $SG$-classical pseudodifferential operators with integer order on $M$. Comparing with the corresponding results in \cite{MP02} and \cite{NI03}, we can give a better estimate for the case $m=m_1= m_2$, both on $\R^n$ as well as, more generally, on manifolds with ends. We first briefly recall the definition of manifold with cylindrical ends, together with the extension on such objects of the concepts of rapidly-decreasing function, temperate distribution, $SG$-calculus and weighted Sobolev space: in this first part of the Section we follow \cite{MP02}, to which the reader can refere for details, with slight modifications in the definitions of the manifold with ends $M$ and of $SG$-classical operator on $M$. The results we obtain hold for manifolds with finitely many cylindrical ends: without loss of generality, to keep notation simple, in the sequel we consider manifolds with a single cylindrical end.

\begin{Def}
	A manifold with a cylindrical end is a triple $(M, X, [f])$, where $M= \mathscr{M} \amalg_C \mathscr{C}$ is a $n$-dimensional
	smooth manifold and
	\begin{enumerate}
		\item[  i)] $\mathscr{M}$ is a smooth manifold, given by $\mathscr{M}=(M_0\setminus D)\cup C$ 
		with a $n$-dimensional smooth compact manifold without boundary $M_0$, $D$ a closed disc of $M_0$ and
		$C\subset D$ a collar neighbourhood of $\partial D$ in $M_0$;
		\item[ ii)] $\mathscr{C}$ is a smooth manifold with boundary $\partial\mathscr{C}=X$, with $X$ diffeomorphic to 
		$\partial D$;
		\item[iii)] $f: [\delta_f, \infty) \times \mathbb{S}^{n-1} \rightarrow \mathscr{C}$, $\delta_f>0$,
				is a diffeomorphism, $f(\{\delta_f \}\times \mathbb{S}^{n-1})=X$ and 
				$f(\{[\delta_f,\delta_f+\varepsilon_f) \}\times \mathbb{S}^{n-1})$, $\varepsilon_f>0$,
				is diffeomorphic to $C$;  
		\item[ iv)] the symbol $\amalg_C$ means that we are gluing $\mathscr{M}$ and $\mathscr{C}$,
				through the identification of $C$ and $f(\{[\delta_f,\delta_f+\varepsilon_f) \}\times \mathbb{S}^{n-1})$;
		\item[  v)] the symbol $[f]$ represents an equivalence class in the set of functions
				\begin{align*}
					\{ g: [\delta_g, \infty) \times \mathbb{S}^{n-1} \rightarrow \mathscr{C} \colon & g \textnormal{ is a diffeomorphism, } \\
						&g(\{\delta_g\}\times \mathbb{S}^{n-1})=X\mbox{ and } \\
						&g([\delta_g, \delta_g+\varepsilon_g) \times \mathbb{S}^{n-1}),
						\mbox{ $\varepsilon_g>0$, is diffeomorphic to $C$}\}
				\end{align*}
				where $f \sim g$ if and only if there exists a diffeomorphism 
				$\Theta \in \textnormal{Diff}(\mathbb{S}^{n-1})$ such that
				\begin{equation}
					\label{econd}
					(g^{-1} \circ f)(\rho, \omega)= (\rho, \Theta(\omega))
				\end{equation}
				for all $\rho\ge \max\{\delta_f, \delta_g\}$ and $\omega \in \mathbb{S}^{n-1}$. 
	\end{enumerate}
\end{Def}

\noindent
We use the following notation: 
\begin{itemize}
\item $U_{\delta_f}= \{x \in \R^n \colon |x|> \delta_f\}$;
\item $ \mathscr{C}_\tau= f([\tau, \infty) \times \mathbb{S}^{n-1})$, where $\tau\ge\delta_f$.
The equivalence condition \eqref{econd} implies that $\mathscr{C}_\tau$ is well defined;
\item $\displaystyle \pi: \R^n\setminus\{0\}\rightarrow  (0, \infty) \times \mathbb{S}^{n-1}: x \mapsto \pi(x)= \Big(|x|, \frac{x}{|x|}\Big)$;
\item $f_\pi= f\circ \pi: \overline{U_{\delta_f}} \rightarrow \mathscr{C}$ is a parametrisation of the end. Let us notice that, setting
$F=g^{-1}_\pi \circ f_\pi$, the equivalence condition \eqref{econd} implies
\begin{equation}
	\label{cambcart}
	F(x)= |x| \; \Theta\Big(\frac{x}{|x|}\Big).
\end{equation}
We also denote the restriction of $f_\pi$ mapping $U_{\delta_f}$ onto $\dot{\mathscr{C}}=\mathscr{C}\setminus X$ by
$\dot{f}_\pi$.
\end{itemize}
The couple $(\dot{\mathscr{C}}, \dot{f}_\pi^{-1})$ is called the exit chart.
If $\mathscr{A}=\{(\Omega_i, \psi_i)\}_{i=1}^N$ is such that the subset $\{(\Omega_i, \psi_i)\}_{i=1}^{N-1}$ is a finite atlas for 
$\mathscr{M}$ and $(\Omega_N, \psi_N)=(\dot{\mathscr{C}}, \dot{f}_\pi^{-1})$, then $M$, with the atlas $\mathscr{A}$, is a $SG$-manifold (see \cite{SH87}): an atlas $\mathscr{A}$ of such kind is called \textit{admissible}. From now on, we restrict the choice of atlases on $M$ to the class of admissible ones. We introduce the following spaces, endowed with their natural topologies:

\[
\begin{split}
\mathscr{S}(U_\delta)&=\left\{u \in C^{\infty}(U_\delta) \colon \forall \alpha, \beta \in \N^n\, \forall \delta'>\delta \,\sup_{x\in U_{\delta^\prime}}|x^\alpha\partial^\beta u(x)|< \infty\right\},\\
\mathscr{S}_0(U_\delta)&= \bigcap_{\delta' \searrow \delta}\{u \in \mathscr{S}(\R_n)\colon \mathrm{supp}\,u \subseteq \overline{U_{\delta'}} \}, \\
\mathscr{S}(M)&= \{ u \in C^{\infty}(M) \colon  u \circ \dot{f}_\pi \in \mathscr{S}(U_{\delta_f}) \mbox{ for any exit map } f_\pi \},\\
\mathscr{S}^\prime(M)&\mbox{ denotes the dual space of $\mathscr{S}(M)$}.
\end{split}
\]

\begin{Def}
The set $SG^{m_1, m_2}(U_{\delta_f})$ consists of all the symbols $a \in C^{\infty}(U_{\delta_f})$ which fulfill \eqref{disSG} for $(x,\xi) \in U_{\delta_f} \times \R^n$ only. Moreover, the symbol $a$ belongs to the subset $SG_{\cl}^{m_1, m_2}(U_{\delta_f})$ if it admits expansions in asymptotic sums of homogeneous symbols with respect to $x$ and $\xi$ as in Definitions \ref{def:sgclass-a} and \ref{def:sgclass-b}, where the remainders are now given by $SG$-symbols of the required order on $U_{\delta_f}$.
\end{Def}

\noindent
Note that, since $U_{\delta_f}$ is conical, the definition of homogeneous and classical symbol on $U_{\delta_f}$ makes sense. Moreover,
the elements of the asymptotic expansions of the classical symbols can be extended by homogeneity to smooth functions on
$\R^n\setminus\{0\}$, which will be denoted by the same symbols.
It is a fact that, given an admissible atlas $\{(\Omega_i, \psi_i)\}_{i=1}^N$ on $M$, there exists a partition of unity $\{\varphi_i\}$ and a
set of smooth functions $\{\chi_i\}$ which are compatible with the $SG$-structure of $M$, that is:
\begin{itemize}
\item $\mathrm{supp}\,\varphi_i\subset\Omega_i$, $\mathrm{supp}\,\chi_i\subset\Omega_i$, $\chi_i\,\varphi_i=\varphi_i$, $i=1,\dots,N$;
\item $|\partial^\alpha(\varphi_N\circ \dot{f}_\pi)(x)|\le C_\alpha \norm{x}^{-|\alpha|}$ and 
$|\partial^\alpha(\chi_N\circ \dot{f}_\pi)(x)|\le C_\alpha \norm{x}^{-|\alpha|}$ for all $x\in U_{\delta_f}$.
\end{itemize}
Moreover, $\varphi_N$ and $\chi_N$ can be chosen so that $\varphi_N\circ \dot{f}_\pi$ and $\chi_N\circ \dot{f}_\pi$
are homogeneous of degree $0$ on $U_\delta$. We denote by
$u^*$ the composition of $u\colon \psi_i(\Omega_i)\subset\R^n\to\C$ with the coordinate patches $\psi_i$,
and by $v_*$ the composition of $v\colon \Omega_i\subset M\to \C$ with $\psi_i^{-1}$, $i=\,\dots,N$.
It is now possible to give the definition of $SG$-pseudodifferential operator on $M$:

\begin{Def}
Let $M$ be a manifold with a cylindrical end. A linear operator $A:\mathscr{S}(M)\to \mathscr{S}'(M) $ is a $SG$-pseudodifferential operator of order $(m_1, m_2)$ on $M$ if, for any admissible atlas $\{(\Omega_i, \psi_i)\}_{i=1}^N$ on $M$ with exit chart
$(\Omega_N,\psi_N)$:
\begin{itemize}
\item[1)] for all $i=1, \ldots, N-1$ and any $\varphi_i,\chi_i\in C_c^{\infty}(\Omega_i)$, 
there exist symbols $a^i(x,\xi) \in S^{m_1}(\psi_i(\Omega_i))$ such that
\[
(\chi_i A \varphi_i \,u^*)_*(x)= \iint e^{i(x-y)\cdot \xi}a^i(x,\xi) u(y) dy dx, \quad u \in C^{\infty}(\psi_i(\Omega_i));
\]
\item[2)] for any $\varphi_N,\chi_N$ of the type described above, there exists a symbol $a^N(x,\xi) \in SG^{m_1, m_2}(U_{\delta_f})$ such that
\[
(\chi_N A \varphi_N\,u^*)_*(x)= \iint e^{i(x-y)\cdot \xi}a^N(x,\xi) u(y) dy dx, \quad u \in \mathscr{S}_0(U_{\delta_f});
\]
\item[3)] $K_A$, the Schwartz kernel of $A$, is such that
\[
K_A \in C^{\infty}\big((M \times M) \setminus \Delta\big) \bigcap \mathscr{S}\big((\dot{\mathscr{C}} \times \dot{\mathscr{C}})\setminus W\big)
\]
where $\Delta$ is the diagonal of $M \times M$ and $W= (\dot{f}_\pi \times \dot{f}_\pi)(V)$ with any conical neighbourhood $V$ of the diagonal of $U_{\delta_f} \times U_{\delta_f}$.
\end{itemize}
\end{Def}

\noindent
The most important local symbol of $A$ is $a^N$, which we will also denote $a^f$, to remind its dependence on the exit chart. Our definition of $SG$-classical operator on $M$ differs slightly from the one in \cite{MP02}:

\begin{Def}
\label{clexit}
Let $A \in L^{m_1, m_2}(M)$. $A$ is a $SG$-classical operator on $M$, and we write $A \in L_{\cl}^{m_1, m_2}(M)$, if
$a^f(x,\xi) \in SG_{\cl}^{m_1, m_2}(U_{\delta_f})$ and the operator $A$, restricted to the manifold
$\mathscr{M}$, is classical in the usual sense.
\end{Def}

\noindent
The principal symbol $a_{m_1, \cdot}$ of a $SG$-classical operator $A\in L^{m_1,m_2}_\cl(M)$ is of course well-defined as
a smooth function on $T^*M$. In order to give an invariant definition of principal symbol with respect to $x$ of an operator $A \in L^{m_1, m_2}_\cl(M)$, the subbundle $T_X^*M=\{(x,\xi) \in T^*M\colon x \in X, \, \xi \in T_x^*M\}$ was introduced. The notions of ellipticity and $\Lambda$-ellipticity can be extended to operators on $M$ as well:

\begin{Def}
Let $A \in L_\cl^{m_1, m_2}(M)$ and let us fix an exit map $f_\pi$. We can define
local objects $a_{m_1-j, m_2-k}, a_{\cdot, m_2-k} $ as
\[
\begin{split}
a_{m_1-j, m_2-k}(\theta, \xi)&= a^f_{m_1-j, m_2-k}(\theta, \xi), \quad \theta \in \mathbb{S}^{n-1}, \,\xi \in \R^n \setminus\{0\},\\
a_{\cdot, m_2-k}(\theta,\xi)&=a^f_{\cdot, m_2-k}(\theta, \xi), \quad \theta \in \mathbb{S}^{n-1},\, \xi \in \R^n.
\end{split}
\]
\end{Def}

\begin{Def}
An operator $A \in L^{m_1, m_2}_\cl(M)$ is elliptic if the principal part of $a^f \in SG^{m_1, m_2}(U_{\delta_f})$ satisfies the $SG$-ellipticity conditions on $U_{\delta_f}\times\R^n$ and the operator $A$, restricted to the manifold
$\mathscr{M}$, is elliptic in the usual sense.
Similarly, $A$ is $\Lambda$-elliptic if the principal part of $a^f$ is $\Lambda$-elliptic on $U_{\delta_f}\times\R^n$ and $A$, restricted to
$\mathscr{M}$, is $\Lambda$-elliptic in the standard sense.
\end{Def}

\begin{Prop}
\label{prop:classinv}
The properties $A \in L^{m_1, m_2}(M)$ and $A\in L^{m_1, m_2}_\cl(M)$, as well as the notions of ellipticity and $\Lambda$-ellipticity, do not depend on the (admissible) atlas. Moreover, the local functions $a_{\cdot, m_2}$ and $a_{m_1, m_2}$ give rise to invariantly defined elements of $C^{\infty}(T_X^*M)$ and $C^{\infty}(T_X^*M\setminus 0)$, respectively.
\end{Prop}
\noindent
Then, with any $A\in L^{m_1,m_2}_\cl(M)$, it is associated an invariantly defined principal symbol in three components
$\sigma(A)=(a_{m_1,.},a_{.,m_2},a_{m_1,m_2})$. Finally, through local symbols given by $p^i(x,\xi)=\norm{\xi}^{s_1}$, $i=1,\dots,N-1$,
and $p^f(x,\xi)=\norm{\xi}^{s_1}\norm{x}^{s_2}$, $s_1,s_2\in\R$,  we get a $SG$-elliptic operator
$\Pi_{s_1,s_2}\in L^{s_1,s_2}_\cl(M)$ and introduce the (invariantly defined) weighted Sobolev spaces $H^{s_1,s_2}(M)$ as
\[
	H^{s_1,s_2}(M)=\{u\in\mathscr{S}^\prime(M)\colon \Pi_{s_1,s_2}u\in L^2(M)\}.
\] 
The properties of the spaces $H^{s_1,s_2}(\R^n)$ extend to $H^{s_1,s_2}(M)$ without any change, as well as the continuous
action of the $SG$-operators mentioned in the Introduction.

We can now formulate a set of hypotheses, analogous to Assumptions \ref{assu}, that imply the existence of $A^z$, $z\in\C$, for
$A\in L^{m_1,m_2}_\cl(M)$:
\begin{Ass}
\label{assusg}
\begin{enumerate}
\item $A \in L^{m_1, m_2}(M)$, with $m_1$ and $m_2$ positive integers;
\item $A$ is $\Lambda$-elliptic with respect to a closed sector $\Lambda$ of the complex plane with vertex at the origin;
\item $A$ is invertible as an operator from $L^2(M)$ to itself;
\item The spectrum of $A$ does not intersect the real interval $(-\infty, 0)$;
\item $A$ is $SG$-classical.
\end{enumerate}
\end{Ass}
\noindent
The definitions of $A^z$ and $\zeta(A, z)$ for such an operator on $M$ follow by the known results on a closed manifold, see \cite{SE67,SH87}, combined, via the $SG$-compatible partition of unity, with similar constructions on the end $\mathscr{C}$: through the exit chart, the latter are achieved by the same techniques used in Sections \ref{sec:zpower} and \ref{seczeta}. Note that, in view of the $SG$-structure on $M$ given by the admissible atlases and the hypotheses, $A^z$ and $\zeta(A, z)$ are invariantly defined on $M$. It is then easy to prove that the properties of $\zeta(A,z)$ extend from $\Rn$ to a general manifold with cylindrical ends. The next Theorem \ref{zetasg} is the global version of Theorem \ref{fzeta} on $M$:
\begin{Thm}
\label{zetasg}
Let $A \in L^{m_1, m_2}(M)$ satisfy Assumptions \ref{assusg}. Then
$\zeta(A,z)$ is holomorphic for $\Re(z)<\min\{-\frac{n}{m_1}, -\frac{n}{m_2}\}$ and can be extended as a meromorphic function with, at most, poles at the points
\[
z_j^1=\frac{j-n}{m_1}, \, j=0, 1, \ldots, \quad z^2_k=\frac{k-n}{m_2}, \, k=0, 1, \ldots
\]
Such poles can be of order two if and only if there exist $j$ and $k$ such that $z^1_j=z^2_k$.
\end{Thm}

\begin{proof}
We have
\begin{equation}
	\label{eq:zsgm}
	\zeta(A,z)= \int_{M} K_{A^z}(y,y)dy=  \int_{\mathscr{M}}K_{A^z}(y,y)dy + \int_{\mathscr{C}\setminus C} K_{A^z}(y,y)dy.
\end{equation}
Since $K_{A^z}(y,y)dy$ has an invariant meaning on $M$, we can perform the computations through an arbitrary admissible atlas $\mathscr{A}=\{(\Omega_i,\psi_i)\}_{i=1}^N$. By the assumptions above, we know that 
$\{(\Omega_i,\psi_i)\}_{i=1}^{N-1}$ is an atlas on $\mathscr{M}$: then, by considerations completely similar to those that hold for compact manifolds without boundary, see, e.g., \cite{SH87}, Ch. 2, we can prove that the first integral in \eqref{eq:zsgm} is a complex function of $z$ with the properties stated above and, at most, poles of the type $z^1_j$, $j=0,1,\dots$ To handle the contribution on 
$\mathscr{C}\setminus C$, we fix an exit map $f_\pi$ and compute the second integral, modulo holomorphic functions of $z$, as
\[
	\int_{\overline{U_{\delta_f+\varepsilon_f}}}K_{(\Op{a^f})^z}(x,x)dx.
\]
We can then show that the remaining assertions on $\zeta(A,z)$ hold true 
by repeating the same steps of the proof of Theorem \ref{fzeta}. 
\end{proof}
We now extend the definition of Wodzicki Residue for $SG$-operators on $\R^n$ to $SG$-operators on $M$ in terms of the zeta function. First of all, choose an admissible atlas and introduce the following functionals on $L^{m_1,m_2}_\cl(M)$, analogous to those defined in \eqref{trnicola}:
\begin{equation}
\label{respsi}
\begin{split}
\TR(A)&= \frac{1}{(2\pi)^{n}} \int_{\mathbb{S}^{n-1}}\int_{\mathbb{S}^{n-1}} a_{-n,-n}(\theta, \theta')d\theta' d\theta, 
\\
\widehat{\mathrm{Tr}}^c_\psi(A)&=  \frac{1}{(2\pi)^{n}}\lim_{\tau \to \infty} \Big[ \int_{ M \setminus \mathscr{C}_\tau}\int_{\mathbb{S}^{n-1}}
a_{-n, \cdot}(x, \theta')d\theta' dx
\\
&\hspace{2.5cm}-(\log\tau)\int_{\mathbb{S}^{n-1}}\int_{\mathbb{S}^{n-1}} a_{-n,-n}(\theta, \theta')d\theta' d\theta
\\
& \hspace{2.5cm}-\sum_{k=0}^{m_2+n-1}\frac{\tau^{m_2-k}}{m_2-k}\int_{\mathbb{S}^{n-1}}\int_{\mathbb{S}^{n-1}}a_{-n, m_2-k}(\theta,\theta')
d\theta' d\theta\Big]\\
\widehat{\mathrm{Tr}}^c_e(A)&=  \frac{1}{(2\pi)^{n}}\lim_{\tau \to \infty} \Big[\int_{\mathbb{S}^{n-1}} \int_{|\xi|\leq \tau}
a_{ \cdot,-n}(\theta, \xi)d\xi d\theta 
\\
&\hspace{2.5cm}-(\log\tau)\int_{\mathbb{S}^{n-1}}\int_{\mathbb{S}^{n-1}} a_{-n,-n}(\theta, \theta')d\theta' d\theta
\\
&\hspace{2.5cm}-\sum_{j=0}^{m_1+n-1}\frac{\tau^{m_1-j}}{(m_1-j)}\int_{\mathbb{S}^{n-1}} \int_{\mathbb{S}^{n-1}} 
a_{ m_1-j,-n} (\theta, \theta') d\theta' d\theta\Big],
\end{split}
\end{equation}
and the angular term, analogous to \eqref{trang},
\begin{equation}
\widehat{\TR}_{\theta}^c(A)=\frac{1}{(2\pi)^{n}}\int_{\mathbb{S}^{n-1}}\int_{\mathbb{S}^{n-1}} \left.\frac{d}{dz}(a_{m_1z-n-m_1, m_2z-n-m_2})\right|_{z=1}(\theta, \theta')d\theta' d\theta.
\end{equation}
Then, by arguments similar to those in the proofs of Theorems \ref{thm:equiv} and \ref{trxxi}, we can prove:
\begin{Thm}
\label{resexit}
Let $A$ be an operator that satisfies Assumptions \ref{assusg} and set
\begin{equation}
	\label{reswo}
	\widehat{\TR}_{x,\xi}(A)=   -\frac{1}{m_1}\widehat{\mathrm{Tr}}^c_\psi(A) - \frac{1}{m_2}\widehat{\mathrm{Tr}}^c_e(A) 
	+\frac{1}{m_1 m_2}\widehat{\TR}_{\theta}^c(A).
\end{equation}
The functionals $\dfrac{\TR(A)}{m_1m_2}$ and $\widehat{\TR}_{x,\xi}(A)$ are the coefficients of the polar parts of order two and of order one of $\zeta(A,z)$ evaluated at $z=1$, respectively.
\end{Thm}

\begin{rem}
The functional $TR$ extends to all $SG$-classical operators on $M$ with integer order, see the argument before \eqref{extres} at the end of Section \ref{seczeta}. In this way, $TR$ turns out to be a trace on the algebra $\mathcal{A}$
of $SG$-classical operators on $M$ with integer order modulo smoothing operators.
\end{rem}

\noindent
We conclude with the proof of the Weyl formulae \eqref{weylnor}. We make use of the following Theorem, immediate corollary of results by J. Aramaki \cite{AR88}:
\begin{Thm}
\label{aramaki}
Let $A$ be a positive selfadjoint operator that satisfies Assumptions \ref{assu} or \ref{assusg}, such that 
$\zeta(A,z)$ is holomorphic on $U=\{z\in\C\colon\Re(z)< z_0+\gamma\} \setminus \{z_0\}$ for some
$\gamma>0$ and admits a pole at $z=z_0<0$. Then\footnote{The Theorem of Aramaki actually requires another assumption on the decay of the $\zeta$-function on vertical strips, which is fulfilled in this case, see \cite{MSS06b}.}, if
\[
\zeta(A,z)+ \sum_{j=1}^p  \frac{A_j}{(j-1)!} \left(\frac{d}{dz}\right)^{j-1} \frac{1}{z-z_0}
\]
is holomorphic on $\{z\in\C\colon\Re(z)<z_0+\gamma\}$, we have, for some $\delta>0$,
\[
N_A(\lambda)= \sum_{j=1}^p \frac{A_j}{(j-1)!} \left.\left(\frac{d}{ds}\right)^{j-1}\left(\frac{\lambda^s}{s}\right)\right|_{s=-z_0}+ O(\lambda^{-z_0-\delta})
\]
as $\lambda\to+\infty$.
\end{Thm}

\begin{Thm}
\label{weylrn}
Let $A$ satisfy Assumptions \ref{assu} or \ref{assusg}, and let $A$ be elliptic, selfadjoint and positive. Then, 
for certain $\delta_i>0$, $i=0,1,2$, the counting function $N_A(\lambda)$ associated with $A$ has the following asymptotic behaviour for $\lambda\to+\infty$:
\begin{eqnarray}
\label{asbe}
\label{m1=m2}
N_A(\lambda) =& C^1_0 \lambda^{\frac{n}{m}}\log \lambda+ C^2_0 \lambda^{\frac {n}{m}}+ O(\lambda^{\frac{n}{m}-\delta_0}) & \mbox{for } m_1=m_2=m \\
\label{m1<m2}
N_A(\lambda) =& C_1 \lambda^{\frac{n}{m_1}}+ O(\lambda^{\frac{n}{m_1}-\delta_1}) & \mbox{for } m_1<m_2\\
\label{m2<m1}
N_A(\lambda) =& C_2 \lambda^{\frac{n}{m_2}}+ O(\lambda^{\frac{n}{m_2}-\delta_2})& \mbox{for }m_1>m_2.
\end{eqnarray}
Moreover, the constants appearing in the above estimates are given by
\begin{align}
\label{constweyla}
C_0^1= \frac{1}{mn}\,\TR(A^{-\frac{n}{m}}), \quad C_0^2&=\widehat{\TR}_{x,\xi}(A^{-\frac{n}{m}})- \frac{1}{n^2}\,\TR(A^{-\frac{n}{m}}),\\
\label{constweylb}
C_1&=\widehat{\TR}_{x,\xi}(A^{-\frac{n}{m_1}}),\\
\label{constweylc}
C_2&=\widehat{\TR}_{x,\xi}(A^{-\frac{n}{m_2}}).
\end{align}
\end{Thm}
\begin{proof}
From Theorems \ref{fzeta} and \ref{zetasg} it follows that
\[
\zeta(A,z)= \int_{1}^{\infty} \lambda^z dN_A(\lambda)
\]
is holomorphic for $\Re(z)<\min \{ -\frac{n}{m_1},-\frac{n}{m_2} \}$. We now examine the three cases separately.
\begin{itemize}
\item[1)]$m_1=m_2=m$.\\ 
	In this case, by the results of Section \ref{seczeta} and Theorem \ref{resexit}, 
	the function $\zeta(A,z)$ is holomorphic for $\Re (z) < -\frac{n}{m}$ and
	\[
	\zeta(A,z)- \frac{A_2}{(z+\frac{n}{m})^2}+ \frac{A_1}{z+ \frac{n}{m}}
	\]
	with
	\[
		A_2= \frac{\TR(A^{-\frac{n}{m}})}{m^2},\quad
		A_1= \frac{n}{m}\,\widehat{\TR}_{x, \xi}(A^{-\frac{n}{m}}),
	\]
	can be extended to an holomorphic function on $\Re(z) < -\frac{n}{m}+\gamma$, $\gamma>0$. 
	\eqref{m1=m2} and \eqref{constweyla} now follow from Theorem \ref{aramaki}.
		\item[2)]$m_1<m_2$.\\
		Similarly, here we have that
		\[
		\zeta(A,z)+ \frac{A_1}{z+ \frac{n}{m_1}}
		\]	
		with
		\[
		A_1= \frac{n}{m_1}\,\widehat{\TR}_{x, \xi}(A^{-\frac{n}{m_1}})
		\]
		can be extended to an holomorphic function on $\Re(z) < -\frac{n}{m_1}+\gamma$, , $\gamma>0$,
		so \eqref{m1<m2} and \eqref{constweylb} follow from Theorem \ref{aramaki}.
		\item[3)]$m_1>m_2$.\\
		The proof is the same of the previous case, exchanging the role of $x$ and $\xi$.
	\end{itemize}
\end{proof}

\bibliographystyle{abbrv}

\begin{thebibliography}{10}

\bibitem{AB02}
B.~Ammann and C.~B{\"a}r.
\newblock The {E}instein-{H}ilbert action as a spectral action.
\newblock In {\em Noncommutative geometry and the standard model of elementary
  particle physics ({H}esselberg, 1999)}, volume 596 of {\em Lecture Notes in
  Phys.}, pages 75--108. Springer, Berlin, 2002.

\bibitem{AR88}
J.~Aramaki.
\newblock On an extension of the {I}kehara {T}auberian theorem.
\newblock {\em Pacific J. Math.}, 133(1):13--30, 1988.

\bibitem{BA10}
U.~Battisti.
\newblock Weyl asymptotics of Bisingular Operators and Dirichlet Divisor Problem.
\newblock {Preprint}.

\bibitem{BC10}
U. ~Battisti and S.~Coriasco.
\newblock A Note on the Einstein-Hilbert action and the Dirac operator on $R^n$.
\newblock{http://arxiv.org/abs/1007.1797}.

\bibitem{BN03}
P.~Boggiatto and F.~Nicola.
\newblock Non-commutative residues for anisotropic pseudo-differential
  operators in {$\mathbb R\sp n$}.
\newblock {\em J. Funct. Anal.}, 203(2):305--320, 2003.

\bibitem{CZ95}
T. ~Christiansen and M. ~Zworski.
\newblock Spectral asymptotics for manifolds with cylindrical ends.
\newblock{\em Ann. Inst. Fourier (Grenoble)}, 45: 251--263, 1995.

\bibitem{CO88}
A. ~ Connes.
\newblock The action functional in noncommutative geometry.
\newblock{\em Comm. Math. Phys.},  117(4): 673--683, 1988. 

\bibitem{CO} 
H. O. Cordes.
\newblock{\em The Technique of Pseudodifferential Operators.} 
\newblock Cambridge Univ. Press, 1995.

\bibitem{ES97}
Y.~V. Egorov and B.-W. Schulze.
\newblock {\em Pseudo-differential operators, singularities, applications},
  volume~93 of {\em Operator Theory: Advances and Applications}.
\newblock Birkh\"auser Verlag, Basel, 1997.

\bibitem{FE96}
B.~V. Fedosov, F.~Golse, E.~Leichtnam, and E.~Schrohe.
\newblock The noncommutative residue for manifolds with boundary.
\newblock {\em J. Funct. Anal.}, 142(1):1--31, 1996.

\bibitem{LO02}
J.~B. Gil and P.~A. Loya.
\newblock On the noncommutative residue and the heat trace expansion on conic
  manifolds.
\newblock {\em Manuscripta Math.}, 109(3):309--327, 2002.

\bibitem{GS95}
G.~Grubb and R.~T. Seeley.
\newblock Weakly parametric pseudodifferential operators and
  {A}tiyah-{P}atodi-{S}inger boundary problems.
\newblock {\em Invent. Math.}, 121(3):481--529, 1995.

\bibitem{GU85}
V.~Guillemin.
\newblock A new proof of {W}eyl's formula on the asymptotic distribution of
  eigenvalues.
\newblock {\em Adv. in Math.}, 55(2):131--160, 1985.

\bibitem{HO07}
L.~H{\"o}rmander.
\newblock {\em The analysis of linear partial differential operators. {III}}.
\newblock Classics in Mathematics. Springer, Berlin, 2007.
\newblock Pseudo-differential operators, Reprint of the 1994 edition.

\bibitem{KW95}
W.~Kalau and M.~Walze.
\newblock Gravity, non-commutative geometry and the {W}odzicki residue.
\newblock {\em J. Geom. Phys.}, 16(4):327--344, 1995.

\bibitem{KA89}
C.~Kassel.
\newblock Le r\'esidu non commutatif (d'apr\`es {M}.\ {W}odzicki).
\newblock {\em Ast\'erisque}, (177-178):Exp.\ No.\ 708, 199--229, 1989.
\newblock S\'eminaire Bourbaki, Vol.\ 1988/89.

\bibitem{KA95}
D.~Kastler.
\newblock The {D}irac operator and gravitation.
\newblock {\em Comm. Math. Phys.}, 166(3):633--643, 1995.

\bibitem{LM02}
R.~Lauter and S.~Moroianu.
\newblock Homology of pseudo-differential operators on manifolds with fibered boundaries.
\newblock {\em J. Reine Angew. Math.},  547:207--234, 2002.

\bibitem{Le99}
M.~Lesch.
\newblock On the noncommutative residue for pseudodifferential operators 
with log-polyhomogeneous symbols.
\newblock {\em  Ann. Global Anal. Geom.}, 17(2):151--187, 1999.

\bibitem{LJ10}
M.~Lesch and C.N.~Jiménez.
\newblock Classification of traces and hypertraces on spaces of classical pseudodifferential operators.
\newblock {Preprint available at http://arxiv.org/abs/1011.3238}.

\bibitem{MP02}
L.~Maniccia and P.~Panarese.
\newblock Eigenvalue asymptotics for a class of md-elliptic {$\psi$}do's on
  manifolds with cylindrical exits.
\newblock {\em Ann. Mat. Pura Appl. (4)}, 181(3):283--308, 2002.

\bibitem{MSS06b}
L.~Maniccia, E.~Schrohe, and J.~Seiler.
\newblock Determinats of classical $sg$-pseudodifferential operators.
\newblock {Preprint avaible
  http://www.ifam.uni-hannover.de/~seiler/artikel/ifam86.pdf}.

\bibitem{MSS06}
L.~Maniccia, E.~Schrohe, and J.~Seiler.
\newblock Complex powers of classical {SG}-pseudodifferential operators.
\newblock {\em Ann. Univ. Ferrara Sez. VII Sci. Mat.}, 52(2):353--369, 2006.

\bibitem{ME} 
R. Melrose. \emph{Geometric scattering theory.} Stanford
Lectures. Cambridge University Press, Cambridge, 1995.

\bibitem{MN} 
R. Melrose and V. Nistor.
\newblock Homology of pseudodifferential operators I. Manifolds with boundary.
\newblock Preprint.

\bibitem{MO08}
S. ~ Moroianu.
\newblock Weyl laws on open manifolds.
\newblock{\em Math. Ann.}, 340:1--21, 2008.

\bibitem{NI03}
F.~Nicola.
\newblock Trace functionals for a class of pseudo-differential operators in
  {$\mathbb R\sp n$}.
\newblock {\em Math. Phys. Anal. Geom.}, 6(1):89--105, 2003.

\bibitem{PA72}
C.~Parenti.
\newblock Operatori pseudodifferenziali in $\mathbb{R}^n$ e applicazioni.
\newblock {\em Ann. Mat. Pura Appl.}, 93:359--389, 1972.

\bibitem{PS07}
S.~Paycha and S.~Scott.
\newblock A Laurent Expansion for Regularized Integrals of Holomorphic Symbols.
\newblock{\em Geom. Fuct. Anal},17:491--536 ,2007. 

\bibitem{PO07}
R.~Ponge.
\newblock Noncommutative residue for Heisenberg manifolds. Applications in CR and contact geometry. 
\newblock {\em  J. Funct. Anal.} 252:399--463, 2007.

\bibitem{PO08}
R.~Ponge.
\newblock Noncommutative geometry and lower dimensional volumes in {R}iemannian
  geometry.
\newblock {\em Lett. Math. Phys.}, 83(1):19--32, 2008.



\bibitem{SC87}
E.~Schrohe.
\newblock Spaces of weighted symbols and weighted {S}obolev spaces on
  manifolds.
\newblock In {\em Pseudodifferential operators ({O}berwolfach, 1986)}, volume
  1256 of {\em Lecture Notes in Math.}, pages 360--377. Springer, Berlin, 1987.

\bibitem{SC88}
E.~Schrohe.
\newblock Complex powers on noncompact manifolds and manifolds with
  singularities.
\newblock {\em Math. Ann.}, 281(3):393--409, 1988.

\bibitem{SC97}
E.~Schrohe.
\newblock Wodzicki's noncommutative residue and traces for operator algebras on
  manifolds with conical singularities.
\newblock In {\em Microlocal analysis and spectral theory (Lucca, 1996)},
  volume 490 of {\em NATO Adv. Sci. Inst. Ser. C Math. Phys. Sci.}, pages
  227--250. Kluwer Acad. Publ., Dordrecht, 1997.

\bibitem{SC98}
B.-W. Schulze.
\newblock {\em Boundary value problems and singular pseudo-differential
  operators}.
\newblock Pure and Applied Mathematics (New York). John Wiley \& Sons Ltd.,
  Chichester, 1998.
\bibitem{SC10}
S.~Scott.
\newblock{\em Traces and Determinants of Pseudodifferential Operators}.
\newblock{Oxford Mathematical Monographs, 2010}.

\bibitem{SE67}
R.~T. Seeley.
\newblock Complex powers of an elliptic operator.
\newblock In {\em Singular Integrals (Proc. Sympos. Pure Math., Chicago, Ill.,
  1966)}, pages 288--307. Amer. Math. Soc., Providence, R.I., 1967.

\bibitem{SH87}
M.~A. Shubin.
\newblock {\em Pseudodifferential operators and spectral theory}.
\newblock Springer Series in Soviet Mathematics. Springer-Verlag, Berlin, 1987.
\newblock Translated from the Russian by Stig I. Andersson.

\bibitem{WO84}
M.~Wodzicki.
\newblock Noncommutative residue. {I}. {F}undamentals.
\newblock In {\em $K$-theory, arithmetic and geometry (Moscow, 1984--1986)},
  volume 1289 of {\em Lecture Notes in Math.}, pages 320--399. Springer,
  Berlin, 1987.





\end{thebibliography}
\def\cftil#1{\ifmmode\setbox7\hbox{$\accent"5E#1$}\else
  \setbox7\hbox{\accent"5E#1}\penalty 10000\relax\fi\raise 1\ht7
  \hbox{\lower1.15ex\hbox to 1\wd7{\hss\accent"7E\hss}}\penalty 10000
  \hskip-1\wd7\penalty 10000\box7}

\end{document}